\documentclass[11pt]{amsart}
\usepackage{amssymb,amsmath,amsfonts,amsthm,enumerate,stmaryrd}
\usepackage{mathrsfs}
\usepackage{graphicx}

\usepackage[left=.75 in, right=.75 in,top=.75 in, bottom=.75 in]{geometry}
\numberwithin{equation}{section}

\usepackage{color}

\providecommand{\abs}[1]{\left\vert#1\right\vert}
\providecommand{\norm}[1]{\left\Vert#1\right\Vert}
\providecommand{\pnorm}[2]{\left\Vert#1\right\Vert_{L^{#2}}}

\providecommand{\Rn}[1]{\mathbb{R}^{#1}}

\providecommand{\br}[1]{\langle #1 \rangle}

\providecommand{\ns}[1]{\norm{#1}^2}

\providecommand{\pns}[2]{\norm{#1}^2_{L^{#2}}}

\def\nab{\nabla}
\def\dt{\partial_t}
\def\hal{\frac{1}{2}}

\def\ls{\lesssim}
\def\p{\partial}

\def\sg{\mathbb{D}}

\def\da{\Delta_{\mathcal{A}}}
\def\naba{\nab_{\mathcal{A}}}
\def\diva{\diverge_{\mathcal{A}}}

\def\H1{{_0}H^1(\Omega)}

\def\sdb{\bar{\mathcal{D}}}
\def\seb{\bar{\mathcal{E}}}

\providecommand{\abs}[1]{\left\vert#1\right\vert}
\providecommand{\norm}[1]{\left\Vert#1\right\Vert}
\providecommand{\pnorm}[2]{\left\Vert#1\right\Vert_{L^{#2}}}

\providecommand{\ns}[1]{\norm{#1}^2}
\providecommand{\pns}[2]{\norm{#1}^2_{L^{#2}}}

\providecommand{\Rn}[1]{\mathbb{R}^{#1}}

\providecommand{\br}[1]{\langle #1 \rangle}

\providecommand{\sdb}[1]{\bar{\mathcal{D}}_{#1}}
\providecommand{\seb}[1]{\bar{\mathcal{E}}_{#1}}

\def\hal{\frac{1}{2}}

\def\ls{\lesssim}

\def\nab{\nabla}
\def\dt{\partial_t}
\def\p{\partial}

\def\da{\Delta_{\mathcal{A}}}
\def\naba{\nab_{\mathcal{A}}}
\def\diva{\diverge_{\mathcal{A}}}

\def\sg{\mathbb{D}}

\def\SH0{\mathcal{H}^0(\Omega)}

\def\A{\mathcal{A}}

\def\D{\mathcal{D}}
\def\E{\mathcal{E}}
\def\F{\mathcal{F}}

\def\H{\mathcal{H}}

\def\N{\mathcal{N}}

\def\dd{\mathrm{d}}

\def\XXint#1#2#3{{\setbox0=\hbox{$#1{#2#3}{\int}$ }
\vcenter{\hbox{$#2#3$ }}\kern-.6\wd0}}

\DeclareMathOperator{\diverge}{div}

\newtheorem{lem}{Lemma}[section]

\newtheorem{prop}[lem]{Proposition}
\newtheorem{thm}[lem]{Theorem}
\newtheorem{remark}[lem]{Remark}

\newtheorem{thm_intro}{Theorem}

\title[Surfactant-driven flows]{Dynamics and stability of surfactant-driven surface waves}

\author{Chanwoo Kim}
\address{
Department of Mathematics\\
University of Wisconsin, Madison\\
Madison, WI 53706 USA
}
\email[C. Kim]{chanwoo.kim@wisc.edu}
\thanks{C. Kim was supported in part by NSF grant DMS 1501031.}

\author{Ian Tice}
\address{
Department of Mathematical Sciences\\
Carnegie Mellon University\\
Pittsburgh, PA 15213, USA
}
\email[I. Tice]{iantice@andrew.cmu.edu}

\begin{document}

\begin{abstract}
In this paper we consider a layer of incompressible viscous fluid lying above a flat periodic surface in a uniform gravitational field.  The upper boundary of the fluid is free and evolves in time.  We assume that a mass of surfactants resides on the free surface and evolves in time with the fluid.  The surfactants dynamics couple to the fluid dynamics by adjusting the surface tension coefficient on the interface and also through tangential Marangoni stresses caused by gradients in surfactant concentration.  We prove that small perturbations of equilibria give rise to global-in-time solutions in an appropriate functional space, and we prove that the solutions return to equilibrium exponentially fast.  In particular this proves the asymptotic stability of equilibria.
\end{abstract}

\maketitle

\section{Introduction }

\subsection{Surfactant-driven flows   }

Surfactants are chemical agents that, when added to a fluid, collect at free interfaces thereby reducing the surface tension.  Variations in the surfactant concentration on the surface also give rise to tangential surface forces called Marangoni forces, which can have a dramatic effect on the surface.  The surfactant dynamics are driven by several effects: absorption and desorption from the free surface, fluid transport along the surface and in the bulk, and both bulk and surface diffusion.  We refer to the books  \cite{ed_bren_was,levich} and the review \cite{Sar} for a more thorough discussion of surfactant physics.  In manufacturing and industrial applications surfactants are a fundamental tool for  stabilizing bubble formation in processes such as foaming, emulsifying, and coating (see the books \cite{myers,rosen} for an exhaustive list of surfactant applications).  Surfactants also play a critical role in preventing the collapse of the lungs during breathing (see  \cite{hills} and the references therein) and are currently being developed as tools to aid in drug delivery in the lungs (see for example \cite{garoff_2, garoff}). 

In this paper we consider  a simple model of surfactants in which a single layer of fluid occupies the three-dimensional domain $\Omega(t)$ with free boundary surface $\Gamma(t)$.  We neglect the effects of absorption and desorption and assume that all of the surfactant is concentrated on $\Gamma(t)$ with density $\hat{c}(\cdot,t):\Gamma(t) \to [0,\infty)$.  The coefficient of surface tension depends on the surfactant concentration $\hat{c}$ via a relation $\sigma = \sigma(\hat{c})$, where we assume that the surface tension function satisfies:
\begin{equation}\label{sigma_assume}
\begin{cases}
\sigma \in C^3([0,\infty)) \\
\sigma \text{ is strictly decreasing.}
\end{cases}
\end{equation}
The latter assumption comes from the fact that surfactants decrease the surface tension in higher concentration, and the former assumption is merely a technical assumption needed for our PDE analysis.  

We will assume that the fluid is viscous and incompressible and that a uniform gravitational field $-g e_3 = (0,0,-g) \in \mathbb{R}^3$ is applied to the fluid.  The fluid and surfactant dynamics then couple through the system of equations (see \cite{ed_bren_was,levich} for derivations and precise definitions of the operators: we will soon reformulate these equations so do not fully define the operators here)   
\begin{equation}\label{surf}
 \begin{cases}
\partial_t u + u \cdot \nabla u + \nab p  = \mu \Delta u - g  e_3  \text{   and } \diverge{u}=0  & \text{in }\Omega(t) \\ 
p\nu -\mu (\nab u + \nab u^T) \nu  =  - \sigma(\hat{c}) \mathcal{H}_{\Gamma(t)} \nu - \nab_{\Gamma(t)} (\sigma(\hat{c})), & \text{on } \Gamma(t) \\ 
 D_t \hat{c} + \hat{c} \diverge_{\Gamma(t)}{u} = \gamma \Delta_{\Gamma(t)} \hat{c} &\text{on }\Gamma(t),
\end{cases}
\end{equation}
where here $\nab_{\Gamma(t)}$ denotes the surface gradient on $\Gamma(t)$,  $D_t$ is a temporal derivative along the flowing surface, $\diverge_{\Gamma(t)}$ is the surface divergence, $\Delta_{\Gamma(t)}$ is the surface Laplacian,  $\gamma > 0$ is the surfactant diffusion constant, and $\mathcal{H}_{\Gamma(t)}$ is twice the mean-curvature operator on $\Gamma(t)$.  The first two equations in \eqref{surf} are the usual incompressible Navier-Stokes equations.  The third equation is the balance of stress on the free surface, and the right-hand side shows that two stresses are generated by the surfactants.  The first term is a normal stress caused by surface curvature, and the second is a tangential stress, known as the Marangoni stress, caused by gradients in the surfactant concentration on the surface.   The fourth equation in \eqref{surf} shows that the surfactant concentration changes to due to flow on the surface and diffusion.

Although surfactant-driven fluid dynamics have been studied numerically by Kwan-Park-Shen \cite{kwan_park_shen} and Xu-Li-Lowengrub-Zhao \cite{xu_li_low_zhao}, there are few rigorous results available in the literature.  The local well-posedness of a two-phase bubble model without gravity and with absorption was proved by Bothe-Pr\"{u}ss-Simonett \cite{bo_pr_si_1}.  The linear stability of the same model was studied by Bothe-Pr\"{u}ss \cite{bo_pr_1}.  The purpose of this paper is to provide rigorous analysis of the model \eqref{surf} for gravity-driven one-phase stratified flows without absorption, which is an important component in the understanding of surfactant-driven flows.

\subsection{Formulation of equations   }

We now specify the equations of motion more precisely.  We will assume that the fluid is horizontally periodic and lies above a flat rigid interface, i.e. that the moving domain $\Omega(t)$ is of the form
\begin{equation*}
\Omega(t) \ :  = \  \{ y \in \Sigma \times \mathbb{R} \;\vert\; -b  < y_3 < \eta(y_1,y_2,t)\},
\end{equation*}
were we assume that $\Sigma : = (L_{1} \mathbb{T}) \times (L_{2} \mathbb{T})$ for $\mathbb{T} = \mathbb{R} / \mathbb{Z}$ and $L_{1},L_{2} >0$ periodicity lengths. The depth of the lower boundary $b>0$ is assumed to be fixed constant, but the upper boundary is a free surface that is the graph of the unknown function $\eta: \Sigma\times \mathbb{R}_{+} \rightarrow \mathbb{R}$.  We will write $\Gamma(t) = \{y_3 = \eta(y_1,y_2,t)\}$ for the free surface of the fluid and $\Sigma_b = \{y_3 =-b\}$ for the fixed bottom surface of the fluid.

Since $\Gamma(t)$ is specified as the graph of $\eta(\cdot,t)$ it is more convenient to redefine the surfactant concentration as a function defined on the cross-section $\Sigma$ rather than on $\Gamma(t)$.  To this end we define the surfactant density function  $\tilde{c} : \Sigma \times \mathbb{R}_+ \to [0,\infty)$ via  $\tilde{c}(x_\ast,t): =  c( x_\ast, \eta (x_\ast,t),t)$.   In what follows we will employ the ``horizontal'' differential operators $\nab_\ast,$ $\diverge_\ast$, and $\Delta_\ast$ (along with writing $x_\ast = (x_1,x_2)$ for $x \in \Rn{3}$), as well as the  versions of the surface differential operators $\nab_\Gamma$, $\diverge_\Gamma$, and $\Delta_\Gamma$ as described in Appendix \ref{app_surf}.  

For each $t$, the fluid is described by its velocity and pressure functions $(u(\cdot,t),p(\cdot,t)) :\Omega(t) \to \mathbb{R}^{3} \times \mathbb{R}$.  Then $(u, p, \eta, \tilde{c})$ satisfy the following system of equations $\Omega(t)$ for $t>0$:
\begin{equation}\label{ns_euler}
 \begin{cases}
\partial_t u + u \cdot \nabla u + \nabla p = \mu \Delta u & \text{in }
\Omega(t) \\
\text{div} \ {u}=0 & \text{in }\Omega(t) \\
\partial_t \eta = u_3 - u_1 \partial_{y_1}\eta - u_2 \partial_{y_2}\eta &
\text{on } \Gamma(t) \\
(p I - \mu \mathbb{D}(u) ) \nu = g \eta \nu -(\sigma H(\eta)   \nu  + \nabla_{\Gamma}\sigma(\tilde{c})) & \text{on } \Gamma(t) \\
\partial_t \tilde{c} +  {u}\cdot \nabla_{*}  \tilde{c} +  \tilde{c} \ \text{div}_{\Gamma} {u} = \gamma \Delta_{\Gamma}  \tilde{c}  & \text{on } \Gamma(t) \\
u = 0 & \text{on } \Sigma_b \\
u(\cdot,0) =u_0, \eta(\cdot,0) = \eta_0, \tilde{c}(\cdot,0) = \tilde{c}_0.
\end{cases}
\end{equation}
Here, we write
\begin{equation*}
\nu= \frac{(- \partial_{y_{1}} \eta, - \partial_{y_{2} } \eta, 1  )}{\sqrt{1+ |\nabla_{*} \eta |^{2}}} 
\end{equation*}
for the outward unit normal on $\Gamma(t)$, $I$ for the $3 \times 3$ identity matrix,  $(\mathbb{D} u)_{ij} = \partial_i u_j + \partial_j u_i$ for the symmetric gradient of $u$, $g>0$ for the strength of gravity, and for $\mu>0$ the viscosity.  Notice also that we have shifted the gravitational forcing to the boundary by redefining the pressure $p \mapsto p + g x_3$.  The tensor $(p I - \mu \mathbb{D}(u))$ is known as the viscous stress tensor.    The mean curvature is denoted by $H(\eta)$ and is given by 
\begin{equation}\label{mean_curvature}
H(\eta)=  \mathrm{div}_* \left(\ \frac{\nabla_* \eta}{\sqrt{1+|\nabla_* \eta|^2}}\right).
\end{equation}

For the sake of convenience we will reduce the number of physical parameters we must keep track of in \eqref{ns_euler} by rescaling in space and time.  By doing so we may assume that $\mu = g =1$ at the cost of relabeling $L_1, L_2, b >0$, and $\gamma>0$.   Throughout the rest of the paper we assume that this scaling has been done in \eqref{ns_euler}.

We assume that the initial surface function $\eta_0$ satisfies the ``zero average'' condition 
\begin{equation}\label{z_avg}
\frac{1}{L_1 L_2} \int_\Sigma \eta_0 =0.
\end{equation}
If it happens that $\eta_0$ does not satisfy \eqref{z_avg} but does satisfy the extra condition that $-b + \br{\eta_0}>0$, where we have written $\br{\eta_0}$ for the left side of \eqref{z_avg}, then it is possible to shift the problem to obtain a solution to \eqref{ns_euler} with $\eta_0$ satisfying \eqref{z_avg}.  Indeed,  we may change 
\begin{equation*}
 y_3 \mapsto y_3 - \br{\eta_0}, \eta \mapsto \eta - \br{\eta_0}, b \mapsto b + \br{\eta_0}, \text{ and } 
p \mapsto p - \br{\eta_0}
\end{equation*}
to find a new solution with the initial surface function satisfying \eqref{z_avg}.  The data $u_0$, $\tilde{c}_0$, and $\eta_0 - \br{\eta_0}$ will satisfy the same compatibility conditions as $u_0,\tilde{c}_0,\eta_0$, and  $b + \br{\eta_0}  >0$, so after renaming we arrive at the above problem with $\eta_0$ satisfying \eqref{z_avg}.  Note that for sufficiently regular solutions to the periodic problem, the condition \eqref{z_avg} persists in time since $\dt \eta = u \cdot \nu \sqrt{1 + \abs{\nab_\ast \eta}^2 }$:
\begin{equation*}
 \frac{d}{dt}  \int_{\Sigma} \eta =  \int_{\Sigma} \dt \eta  = \int_{\Gamma(t) } u \cdot \nu = \int_{\Omega(t)} \diverge{u} = 0.
\end{equation*}

\subsection{Energy-dissipation structure, equilibria,  and conservation of surfactant mass}

Upon taking the dot product of the first equation in \eqref{ns_euler} with $u$, integrating by parts over $\Omega$, and employing all but the fifth equation in \eqref{ns_euler} we may deduce (see Proposition \ref{ed_geometric} for details of a similar calculation) the energy-dissipation equation
\begin{equation*}
 \frac{d}{dt} \left(  \int_{\Omega(t)} \hal \abs{u}^2 +  \int_{\Sigma} \frac{1}{2} \abs{\eta}^2 \right) + \int_{\Omega(t)} \hal \abs{\sg u}^2 = \int_{\Gamma(t)}  \sigma(\tilde{c}) H u \cdot \nu + \nab_\Gamma \sigma(\tilde{c}) \cdot u  
\end{equation*}
Now we use the fact that $\Gamma(t) = \{x_3 = \eta(x_\ast,t)\}$ in conjunction with Proposition \ref{surf_ibp} to write  
\begin{equation*}
\begin{split}
\int_{\Gamma(t)}  \sigma(\tilde{c}) H u \cdot \nu + \nab_\Gamma \sigma(\tilde{c}) \cdot u  & = 
\int_{\Sigma} \left( \sigma(\tilde{c}) H u \cdot \nu + \nab_\Gamma \sigma(\tilde{c}) \cdot u \right) \sqrt{1+\abs{\nab_\ast \eta}^2} 
\\
&= \int_\Sigma  - \sigma(\tilde{c}) \diverge_\Gamma ( u )   \sqrt{1+\abs{\nab_\ast \eta}^2}.
\end{split}
\end{equation*}
From this we conclude that 
\begin{equation}\label{intro_ed_u}
  \frac{d}{dt} \left(  \int_{\Omega(t)} \hal \abs{u}^2 +  \int_{\Sigma} \frac{1}{2} \abs{\eta}^2 \right) + \int_{\Omega(t)} \hal \abs{\sg u}^2 =\int_\Sigma  - \sigma(\tilde{c}) \diverge_\Gamma ( u)    \sqrt{1+\abs{\nab_\ast \eta}^2}.
\end{equation}
The term on the right does not admit a good sign and in fact shows that energy is exchanged between the fluid and the surfactant.  

Given the parabolic form of the $\tilde{c}$ equation in \eqref{ns_euler}, we might expect to find a good energy-dissipation relation for $\tilde{c}$ by multiplying the equation by $\tilde{c} \sqrt{1+\abs{\nab_\ast \eta}^2}$ and integrating by parts.   Proposition \ref{surf_c_ev} applied with $f(z) = z^2/2$ shows that this yields the equation
\begin{equation}\label{intro_ed_c}
\frac{d}{dt} \int_\Sigma \hal \abs{\tilde{c}}^2 \sqrt{1+\abs{\nab_\ast \eta}^2}  + \int_\Sigma \gamma \abs{\nab_\Gamma \tilde{c}}^2\sqrt{1+\abs{\nab_\ast \eta}^2}   = \int_\Sigma -\hal \abs{\tilde{c}}^2 \diverge_\Gamma u     \sqrt{1+\abs{\nab_\ast \eta}^2}
\end{equation}
Once again we see that the term on the right does not admit a good sign and shows that energy is transferred from the fluid to the surfactant.  We might hope initially that the same interaction term with different signs would appear in both \eqref{intro_ed_u} and \eqref{intro_ed_c}, so that upon summing we would get a clean energy-dissipation relation for the total system including the fluid and the surfactant.  Evidently, though, this is not the case, so it is not immediately clear that the problem \eqref{ns_euler} admits a useful energy-dissipation structure.

To get around this problem we will look at the evolution of a more complicated quantity than $\abs{\tilde{c}}^2$.  For any $r \in (0,\infty)$ we define the function $\xi_r : [0,\infty) \to \mathbb{R}$ via 
\begin{equation}\label{xi_def}
 \xi_r(x) = x \left( \frac{\sigma(r)}{r} - \int_{r}^x \frac{\sigma(z)}{z^2} dz  \right).
\end{equation}
The inclusion $\sigma \in C^3$ from \eqref{sigma_assume} tells us that $\xi_r \in C^4$.  Integration by parts and differentiation reveals that 
\begin{equation*}
 \xi_r(x) = \sigma(x) - x \int_{r}^x \frac{\sigma'(z)}{z} dz \text{ and } \xi_r'(x) = -\int_{r}^x \frac{\sigma'(z)}{z} dz.
\end{equation*}
From these and the fact that $\sigma \ge 0$ is decreasing we deduce that $\xi_r$ obeys the following properties:
\begin{equation}\label{xi_properties}
\begin{cases}
\xi_r(x) - x \xi_r'(x) =\sigma(x) \text{ for }x \in [0,\infty) \\
\xi_r''(x) = -\sigma'(x)/x \ge 0 \text{ for }x \in [0,\infty) \\
\xi_r \text{ is strictly convex on } [0,\infty) \\
\xi_r \text{ is strictly decreasing on } [0,r) \\
\xi_r \text{ is strictly increasing on } (r,\infty) \\
\xi_r(x) \ge \sigma(r) \ge 0 \text{ for }x \in [0,\infty) \\
\xi_r(x) =\sigma(r) \Leftrightarrow x=r.
\end{cases}
\end{equation}

Now that we have introduced $\xi_r$ in \eqref{xi_def} we employ Proposition \ref{surf_c_ev} and the first equation in \eqref{xi_properties} to see that 
\begin{multline}\label{intro_ed_xi}
\frac{d}{dt} \int_\Sigma \xi_r(\tilde{c}) \sqrt{1+\abs{\nab_\ast \eta}^2} + \int_\Sigma \gamma \xi_r''(\tilde{c}) \abs{\nab_\Gamma \tilde{c}}^2  \sqrt{1+\abs{\nab_\ast \eta}^2} = \int_\Sigma \left( \xi_r(\tilde{c}) - \xi_r'(\tilde{c}) \tilde{c} \right) \diverge_\Gamma u     \sqrt{1+\abs{\nab_\ast \eta}^2} \\
= \int_\Sigma \sigma(\tilde{c}) \diverge_\Gamma u     \sqrt{1+\abs{\nab_\ast \eta}^2}. 
\end{multline}
Thus, upon summing \eqref{intro_ed_u} and \eqref{intro_ed_xi} we find that
\begin{equation}\label{intro_ed_combo}
\frac{d}{dt} \left(  \int_{\Omega(t)} \hal \abs{u}^2 +  \int_{\Sigma} \frac{1}{2} \abs{\eta}^2 + \xi_r(\tilde{c}) \sqrt{1+\abs{\nab_\ast \eta}^2}    \right) + \int_{\Omega(t)} \hal \abs{\sg u}^2 + \int_\Sigma \gamma \xi_r''(\tilde{c}) \abs{\nab_\Gamma \tilde{c}}^2  \sqrt{1+\abs{\nab_\ast \eta}^2} =0.
\end{equation}
Since $\xi_r, \xi_r'' \ge 0$, this reveals that the problem \eqref{ns_euler} does in fact admit a nice energy-dissipation equation in which all of the energy functionals are non-negative.

We can employ \eqref{intro_ed_combo} to find the equilibria of the problem \eqref{ns_euler}.  Indeed, if we assume a time-independent ansatz in \eqref{ns_euler} then \eqref{intro_ed_combo} tells us that $\sg u =0$ and $\xi_r''(\tilde{c}) \abs{\nab_\Gamma \tilde{c}}^2 =0$.  Since $u=0$ on $\Sigma_b$ we deduce that $u=0$ in $\Omega$, and since $\sigma' <0$ (due to \eqref{sigma_assume}) and $\xi''(x) = -\sigma'(x) /x$ we deduce that $\tilde{c}=c_0 \in (0,\infty)$, where we avoid non-positive solutions for obvious physical reasons.  Using these and \eqref{z_avg} in conjunction with the first and fourth equations in \eqref{ns_euler} then shows that $p=0$ and $\eta =0$.  We thus deduce that \eqref{ns_euler} admits a one-parameter family of equilibrium solutions $u=0$, $p=0$, $\eta=0$, and $\tilde{c} = c_0 \in (0,\infty)$.  The parameter $c_0$ may be chosen, for instance, by assuming that there is a fixed surfactant mass $M_{surf} >0$:
\begin{equation*}
 M_{surf} = \int_{\Sigma} \tilde{c} \sqrt{1 + \abs{\nab \eta}^2} = \int_\Sigma c_0 = \abs{\Sigma } c_0 = L_1 L_2 c_0.
\end{equation*}
In this way we may view the equilibrium solution as being uniquely determined by the mass of surfactant present on the flat equilibrium surface.

The identity \eqref{intro_ed_combo} is valid for any choice of $r \in (0,\infty)$, but when we wish to consider solutions near the equilibrium configuration $u=0$, $p=0$, $\eta=0$, $\tilde{c} = c_0 \in (0,\infty)$ we should choose $r = c_0$.  In this case  \eqref{intro_ed_combo} implies that the energy (the term in the time derivative) does not increase in time, which already suggests that the problem \eqref{ns_euler} should admit stable equilibria.  However, it is not at all obvious from the form of the dissipation functional (the terms outside the time derivative in \eqref{intro_ed_combo}) that it is coercive over the energy functional, and thus it is not clear from the energy-dissipation relation \eqref{intro_ed_combo} that the equilibrium is asymptotically stable.

Proposition \ref{surf_c_ev} also allows us to deduce a basic conservation law for $\tilde{c}$.  Indeed, we use $f(z) = z$ there to find that 
\begin{equation*}
 \frac{d}{dt} \int_\Sigma \tilde{c} \sqrt{1+\abs{\nab_\ast \eta}^2} = 0.
\end{equation*}
The physical interpretation of this is that the overall surfactant mass present on the moving interface does not change in time.  We will always assume that the initial data $\tilde{c}_0 = \tilde{c}(\cdot,0)$ and $\eta_0 = \eta(\cdot,0)$ are related to the equilibrium surfactant concentration $c_0$ via 
\begin{equation}\label{conserv_c}
c_0 := \frac{1}{\abs{\Sigma}} \int_{\Sigma} \tilde{c}_{0}  \sqrt{1+\abs{\nab_\ast \eta_0}^2}.  
\end{equation}
In other words, the initial data $\eta_0, \tilde{c}_0$ uniquely determine which equilibrium solution is a possible candidate for the limit as $t \to \infty$ to solutions to \eqref{ns_euler}.

\subsection{Reformulation  }

In order to work in a fixed domain, we employ a frequently used transformation: see \cite{B,GT1,GT2,WTK}. We consider the fixed equilibrium domain
\begin{equation}\notag
\Omega:= \{x \in \Sigma \times \mathbb{R} \; \vert\;  -b < x_3 < 0  \},
\end{equation}
for which we will write the coordinates as $x\in \Omega$.  We will think of $\Sigma$ as the upper boundary of $\Omega$, and we will write $\Sigma_b := \{x_3 = -b\}$ for the  lower boundary.  We continue to view $\eta$ as a function on $\Sigma \times \mathbb{R}^+$.  We then define 
\begin{equation*}
 \bar{\eta}:= \mathcal{P} \eta = \text{harmonic extension of }\eta \text{ into the lower half space},
\end{equation*}
where $\mathcal{P}$ is as defined by \eqref{poisson_def_per}.  The harmonic extension $\bar{\eta}$ allows us to flatten the coordinate domain via the mapping
\begin{equation}\label{mapping_def}
 \Omega \ni x \mapsto   (x_1,x_2, x_3 +  \bar{\eta}(x,t)(1+ x_3/b(x_1,x_2) )) = \Theta(x,t) = (y_1,y_2,y_3) \in \Omega(t).
\end{equation}
Note that $\Theta(\Sigma,t) = \{ y_3 = \eta(y_1,y_2,t) \} = \Gamma(t)$ and $\Theta(\cdot,t)\vert_{\Sigma_b} = Id_{\Sigma_b}$, i.e. $\Theta$ maps $\Sigma$ to the free surface and keeps the lower surface fixed.   We have
\begin{equation*}
 \nabla  \Theta=
\begin{pmatrix}
 1 & 0 & 0 \\
 0 & 1 & 0 \\
 A & B & J
\end{pmatrix}
\text{ and }
 \mathcal{A} := (\nabla  \Theta^{-1})^T =
\begin{pmatrix}
 1 & 0 & -A K \\
 0 & 1 & -B K \\
 0 & 0 & K
\end{pmatrix},
\end{equation*}
for
\begin{equation*}\label{ABJ_def}
\begin{split}
A &= \partial_1 \bar{\eta} \tilde{b} -( x_3 \bar{\eta} \partial_1 b )/b^2,\;\;\;  B = \partial_2 \bar{\eta} \tilde{b} -( x_3 \bar{\eta} \partial_2 b )/b^2,  \\
J &=  1+ \bar{\eta}/b + \partial_3 \bar{\eta} \tilde{b},  \;\;\; K = J^{-1}, \\
\tilde{b}  &= (1+x_3/b).
\end{split}
\end{equation*}
Here $J = \mathrm{det}{\nabla \Theta}$ is the Jacobian of the coordinate transformation.

Now we define the transformed quantities as (abusing notation slightly)
\begin{equation}\notag
u(t,x) := u(t,\Theta(t,x)), \ p(t,x) := p(t,\Theta(t,x)), \ 
\tilde{c}(t,x_{*}) = \tilde{c}(t,\Theta_\ast(t,x_{*}, 0)).
\end{equation}
In the new coordinates we rewrite (\ref{ns_euler}) as  
\begin{equation}\label{ns_geometric_non-pert} 
\begin{cases}
\partial_t  u  - \partial_{t}  \bar{\eta} \tilde{b} K \partial_{3} u +  {u} \cdot \nabla_{\mathcal{A}}   u
+ \text{div}_{\mathcal{A}} S_{\mathcal{A}} (p,u)= 0
 & \text{in }
\Omega  \\
\text{div}_{\mathcal{A}}    {u}=0 & \text{in }\Omega  \\
\partial_t  {\eta}- {u} \cdot \mathcal{N}  = 0 &
\text{on } \Sigma \\
S_{\mathcal{A}}( {p}, {u}) \mathcal{N} = { \eta} \mathcal{N} - \sigma(\tilde{c}) H \mathcal{N} - \sqrt{1+ |\nabla_{*}{\eta}|^{2}}\sigma^{\prime} (\tilde{c}) \nabla_{\Gamma} \tilde{c}  & \text{on } \Sigma \\
\partial_t \tilde{c} + {u}\cdot \nabla_{*} \tilde{c} +  \tilde{c} \ \text{div}_{\Gamma}  {u}- \gamma \Delta_{\Gamma} \tilde{c}  =  0& \text{on }\Sigma \\
 {u} = 0 & \text{on } \Sigma_{b}.
\end{cases}
\end{equation}

Here we have written the differential operators $\naba$, $\diva$, and $\da$ with their actions given by $(\naba f)_i := \mathcal{A}_{ij} \p_j f$, $\diva X := \mathcal{A}_{ij}\p_j X_i$, and $\da f = \diva \naba f$ for appropriate $f$ and $X$; for $u\cdot \naba u$ we mean $(u \cdot \naba u)_i := u_j \mathcal{A}_{jk} \p_k u_i$.  We have also  written $S_{\mathcal{A}}({p}, {u}) : = ( {p}I - \mathbb{D}_{\mathcal{A}} {u})$ for $(\mathbb{D}_{\mathcal{A}}  {u})_{ij} = \sum_{k} (\mathcal{A}_{ik} \partial_{k}  {u}_{j} +  \mathcal{A}_{jk} \partial_{k}  {u}_{i})$.    Also, $\mathcal{N} = (- \nabla_{*} \eta, 1)$ denotes the non-unit normal on $\Gamma(t)$,  $\nab_\Gamma$, $\nab_\ast$, $\diverge_\Gamma$ are the differential operators defined in Appendix \ref{app_surf}, and $H$ is still of the form \eqref{mean_curvature}.

\subsection{Perturbation form  }

It will be convenient to reformulate \eqref{ns_geometric_non-pert} in a perturbative form for the surfactant concentration, i.e. to view the solution as perturbed around the equilibrium configuration.  To this end we define the perturbation  
\begin{equation}\label{c_def}
c = \tilde{c}- c_0. 
\end{equation}
Then $(u,p,\eta,c)$ satisfy   
\begin{equation}\label{ns_geometric}
\begin{cases}
\partial_t  u  - \partial_{t}  \bar{\eta} \tilde{b} K \partial_{3} u +  {u} \cdot \nabla_{\mathcal{A}}   u
+ \text{div}_{\mathcal{A}} S_{\mathcal{A}} (p,u)= 0
 & \text{in }
\Omega  \\
\text{div}_{\mathcal{A}}    {u}=0 & \text{in }\Omega  \\
\partial_t  {\eta}- {u} \cdot \mathcal{N}  = 0 &
\text{on } \Sigma \\
S_{\mathcal{A}}( {p}, {u}) \mathcal{N} = { \eta} \mathcal{N} - \sigma(c + c_0 ) H \mathcal{N} - \sqrt{1+ |\nabla_{*}{\eta}|^{2}}\sigma^{\prime} (c+ c_0 ) \nabla_{\Gamma} c  & \text{on } \Sigma \\
\partial_t c + {u}\cdot \nabla_{*} c +  (c + c_0) \ \text{div}_{\Gamma}  {u}- \gamma \Delta_\Gamma c  =  0& \text{on }\Sigma \\
 {u} = 0 & \text{on } \Sigma_{b}.
\end{cases}
\end{equation}
Throughout the rest of the paper we will employ the notation
\begin{equation}\label{surf_ten_pert}
 \sigma_0 = \sigma(c_0) \text{ and } \sigma'_0 = \sigma'(c_0).
\end{equation}

\section{Main results and discussion }

\subsection{Main results  }

In order to state our main results we first define the energy and dissipation functionals that we shall use in our analysis.  We define the energy via
\begin{equation}\label{def_E}
\begin{split}
\mathcal{E} &: = \| u \|_{H^{2}(\Omega)}^{2} + \| \partial_{t}  u \|_{H^{0}(\Omega)}^{2}
+ \| p  \|_{H^{1}(\Omega)}^{2} + \| \eta \|_{H^{3}(\Sigma)}^{2} + \| \partial_{t} \eta \|_{H^{\frac{3}{2}}(\Sigma)}^{2}
+ \| \partial_{t}^{2} \eta \|_{H^{- \frac{1}{2}}(\Sigma)}^{2}  \\
& \ \  \ \ \  + \| c \| _{H^{2}(\Sigma)}^{2} + \| \partial_{t}  c \|_{H^{0}(\Sigma)}^{2},
\end{split}\end{equation}
and we define the dissipation as
\begin{equation}\begin{split}\label{def_D}
\mathcal{D} & := \| u \|_{H^{3}(\Omega)}^{2} + \| \partial_{t} u\|_{H^{1}(\Omega)}^{2} 
  + \| \eta \|_{H^{\frac{7}{2}}(\Sigma) }^{2}+   \| \partial_{t}\eta \|_{H^{\frac{5}{2}}(\Sigma)}^{2}
+  \| \partial_{t}^{2}\eta \|_{H^{\frac{1}{2}}(\Sigma)}^{2}
\\
& \ \   \   + \| p \|_{H^{2}(\Omega)}^{2} 
+ \|  c\|_{H^{3}(\Sigma)}^{2} + \| \partial_{t} c  \|_{H^{1}(\Sigma)}^{2}.
\end{split}
\end{equation}
Here the spaces $H^s$ denote the usual $L^2-$based Sobolev spaces of order $s$.  

Our main result is an a priori estimate for solution to \eqref{ns_geometric}.

\begin{thm_intro}[Proved later in Theorem \ref{aprioris}]\label{aprioris_intro}
Suppose that $(u,p,\eta,c)$ solves \eqref{ns_geometric} on the temporal interval $[0,T]$.  Let $\E$ and $\D$ be as defined in \eqref{def_E} and \eqref{def_D}. Then there exists a universal constant $0 < \delta_\ast$ (independent of $T$) such that if 
\begin{equation*}
 \sup_{0\le t \le T} \E(t) \le \delta_\ast \text{ and } \int_0^T \D(t) dt < \infty,
\end{equation*}
then
\begin{equation*}
\sup_{0\le t \le T} e^{\lambda t} \E(t) + \int_0^T \D(t)dt \ls \E(0)
\end{equation*}
for all $t \in [0,T]$, where $\lambda >0$ is a universal constant.
\end{thm_intro}

This result says that if solutions exist for which the energy functional $\mathcal{E}$ remains small and the dissipation functional is integrable in time, then in fact we have much stronger information: the energy decays exponentially and the integral of the dissipation is controlled by the initial energy.  In order for this result to be useful we must couple it with a local existence result.  By now it is well-understood how to construct local-in-time solutions for solutions to problems of the form \eqref{ns_geometric} once the corresponding a priori estimates are understood: we refer for instance to \cite{GT1,WTK,Wu} for local existence results in spaces determined by energies and dissipations of the form \eqref{def_E} and \eqref{def_D}.  Consequently, in the interest of brevity, we will not attempt to prove a local existence result in the present paper.  Instead we will simply state the result that one can prove by modifying the known methods in straightforward ways.

Given the initial data $u_0,\eta_0,\tilde{c}_0$, we need to construct the initial data $\partial_t u(\cdot,0)$, $\partial_t \eta(\cdot,0)$, $\dt c(\cdot,0)$, and $p(\cdot,0)$.  To construct these we  require a compatibility condition for the data.  To state this properly we  define the orthogonal projection onto the tangent space of the surface $\Gamma(0) = \{x_3 = \eta_0(x_\ast)\}$ according to 
\begin{equation*}
 \Pi_0 v = v - (v\cdot \N_0) \N_0 \abs{\N_0}^{-2}
\end{equation*}
for $\N_0 = (-\p_1 \eta_0,-\p_2 \eta_0,1)$.  Then the compatibility conditions for the data read
\begin{equation}\label{compat_cond}
\begin{cases}
\Pi_0 (\sg_{\mathcal{A}_0} u_0 \N_0) - \sqrt{1+\abs{\nab_\ast \eta_0}^2} \sigma'(\tilde{c}_0) \nab_{\Gamma_0} \tilde{c}_0 =0 &\text{on }\Sigma \\
\diverge_{\mathcal{A}_0} u_0 =0 & \text{in }\Omega \\
u_0 =0 & \text{on } \Sigma_b,
\end{cases}
\end{equation}
where here $\mathcal{A}_0$ and $\Gamma_0$ are determined by $\eta_0$.  To state the local result we will also need to define  $\H1 := \{ u \in H^1(\Omega) \;\vert\;  u\vert_{\Sigma_b}=0\}$ and 
\begin{equation}\label{X_space_def}
\mathcal{X}_T = \{u \in L^2([0,T];\H1)\;\vert\; \diverge_{\mathcal{A}(t)} u(t)=0 \text{ for a.e. } t\}. 
\end{equation}

Having stated the compatibility conditions, we can now state the local existence result.

\begin{thm_intro}\label{lwp_intro}
Let $u_0 \in H^2(\Omega)$,   $\eta_0 \in H^3(\Sigma)$, and $\tilde{c}_0 \in H^2(\Sigma)$, and assume that $\eta_0$ and $\tilde{c}_0$ satisfy \eqref{z_avg} and \eqref{conserv_c}, where $c_0 \in (0,\infty)$ is a fixed equilibrium surfactant concentration.  Further assume that the initial data satisfy the compatibility conditions of \eqref{compat_cond}. Let $T >0$. Then there exists a universal constant $\kappa >0$ such that if 
\begin{equation*}
 \ns{u_0}_{H^2(\Omega)} + \ns{\eta_0}_{H^3(\Sigma)} + \ns{\tilde{c}_0 - c_0}_{H^2(\Sigma)} \le \kappa, 
\end{equation*}
then there exists a unique (strong) solution $(u,p,\eta,c)$ to \eqref{ns_geometric} on the temporal interval $[0,T]$ satisfying the estimate 
\begin{equation}\label{lwp_intro_0}
 \sup_{0\le t \le T} \E(t) + \int_0^T \D(t)dt   + \int_0^T \ns{\dt^2 c(t)}_{H^{-1}(\Sigma)}dt + \ns{\dt^{2N+1} u}_{(\mathcal{X}_T)^*} \ls \E(0).
\end{equation}
Moreover, $\eta$ is such that the mapping $\Theta(\cdot,t)$, defined by \eqref{mapping_def}, is a $C^{1}$ diffeomorphism for each $t \in [0,T]$. 
\end{thm_intro}

\begin{remark}
All of the computations involved in the a priori estimates that we develop in this paper are justified by Theorem \ref{lwp_intro}.   
\end{remark}

With local existence, Theorem \ref{lwp_intro}, and a priori estimates, Theorem \ref{aprioris_intro}, in hand, we may couple them to deduce a global existence and decay result. 

\begin{thm_intro}[Proved later in Section \ref{sec_mains}]\label{gwp_intro}
Let $u_0 \in H^2(\Omega)$,   $\eta_0 \in H^3(\Sigma)$, and $\tilde{c}_0 \in H^2(\Sigma)$, and assume that $\eta_0$ and $\tilde{c}_0$ satisfy \eqref{z_avg} and \eqref{conserv_c}, where $c_0 \in (0,\infty)$ is a fixed equilibrium surfactant concentration.  Further assume that the initial data satisfy the compatibility conditions of \eqref{compat_cond}.   Then there exists a universal constant $\kappa >0$ such that if 
\begin{equation}\label{gwp_intro_00}
 \ns{u_0}_{H^2(\Omega)} + \ns{\eta_0}_{H^3(\Sigma)} + \ns{\tilde{c}_0 - c_0}_{H^2(\Sigma)} \le \kappa, 
\end{equation}
then there exists a unique (strong) solution $(u,p,\eta,c)$ to \eqref{ns_geometric} on the temporal interval $[0,\infty)$ satisfying the estimate 
\begin{equation}\label{gwp_intro_01}
\sup_{t \ge 0 } e^{\lambda t} \E(t) + \int_0^\infty \D(t)dt \ls \E(0),
\end{equation}
where $\lambda >0$ is a universal constant.  
\end{thm_intro}

\begin{remark}
 Theorem \ref{gwp_intro} can be interpreted as an asymptotic stability result: the equilibria $u=0$, $p=0$, $\eta=0$, $\tilde{c}=c_0$ are asymptotically stable, and solutions return to equilibrium exponentially fast.
\end{remark}

\begin{remark}
The surface function $\eta$ is sufficiently small to guarantee that the mapping $\Theta(\cdot,t)$, defined in \eqref{mapping_def}, is a diffeomorphism for each $t\ge 0$.  As such, we may change coordinates to $y \in \Omega(t)$ to produce a global-in-time, decaying solution to \eqref{ns_euler}.
\end{remark}

It is worth comparing the result of Theorem \ref{gwp_intro} to what is known about horizontally-periodic surfactant-free viscous surface waves with and without surface tension.  Without surfactants but with a fixed surface tension $\sigma >0$, the problem \eqref{ns_geometric} admits small-data global-in-time solutions that decay to equilibrium exponentially fast, as was proved in \cite{nishida_1}.  If surface tension is neglected, i.e. $\sigma=0$, then again small-data solutions exist for all time, but they decay at an algebraic rate determined by the regularity of the data, as proved in \cite{GT2}.  Thus we see that although surfactants dynamically adjust the surface tension, the behavior of solutions is comparable to solutions to the problem with a fixed surface tension.

\subsection{Summary of methods and plan of paper  }

Our analysis employs a nonlinear energy method based on a higher-regularity modification of the basic energy-dissipation equation \eqref{intro_ed_combo} for solutions to \eqref{ns_geometric}.  Below we will summarize the steps needed to implement this method and how they relate to the organization of the paper.

\textbf{Horizontal energy estimates:} Certainly the form of \eqref{intro_ed_combo} is tied to the choice of boundary conditions in \eqref{ns_geometric}, and so we can only appeal to \eqref{intro_ed_combo} to gain control of derivatives of solutions in directions that are compatible with the boundary conditions.  The choice of $\Omega$ dictates that these are precisely the horizontal spatial directions, corresponding to the operators $\p_1$ and $\p_2$, and the temporal direction, corresponding to $\dt$.  We will get estimates for one temporal and up to two spatial horizontal derivatives; this choice comes from the parabolic scaling of the Navier-Stokes equations, which dictates that each temporal derivative behaves like two spatial derivatives.  Our choice for this number of derivatives comes from the ability to close our energy method: we cannot close with fewer than one temporal derivative, and we get no improvement with more.

The differential operators in \eqref{ns_geometric} do not commute with the operators $\p_1,\p_2,\dt$, so we do not arrive at a ``horizontal'' energy-dissipation equation of exactly the same form as \eqref{intro_ed_combo}.  Indeed, there are nonlinear interaction terms that lead us to an equation of the form (roughly speaking) 
\begin{equation}\label{summary_1}
\frac{d}{dt} \bar{\E} + \bar{\D} = \mathcal{I},
\end{equation}
where $\bar{\E}$ and $\bar{\D}$ are the ``horizontal'' energy and dissipation, respectively, and $\mathcal{I}$ denotes the nonlinear interaction term. 

In order to make $\mathcal{I}$ manageable within our functional framework we are forced to employ different strategies in dealing with spatial derivatives than in dealing with temporal derivatives.  Indeed, for temporal derivatives we must take advantage of certain ``geometric'' identities related to the operators in \eqref{ns_geometric}, whereas for spatial derivatives it is more convenient to shift to constant-coefficient operators for which the connection to the boundary geometry is obfuscated.  These strategies are developed in Section \ref{sec_ed}, and it is here where we make precise the form of the terms appearing in $\mathcal{I}$.

\textbf{Nonlinear estimates:} The next step in our nonlinear energy method is to estimate the terms appearing in the nonlinearity $\mathcal{I}$.  It is not enough for us to be able to control $\mathcal{I}$ within our functional setting: we must have estimates of a particular structural form in order to be able to effectively combine the estimates with \eqref{summary_1}.  This structure roughly requires that we be able to ``absorb'' $\mathcal{I}$ into the dissipation on the left side of \eqref{summary_1}.  More precisely, we seek to prove that (again, roughly speaking)
\begin{equation}\label{summary_2}
 \abs{\mathcal{I}} \ls \sqrt{\E} \D.
\end{equation}
Note here that $\E$ and $\D$ are the full energy and dissipation given by \eqref{def_E} and \eqref{def_D}, and not their horizontal counterparts $\bar{\E}$ and $\bar{\D}$.  This is by necessity: the nonlinear terms in $\mathcal{I}$ cannot be controlled simply in terms of $\bar{\E}$ and $\bar{\D}$.  Structural estimates of the form \eqref{summary_2} are derived in Section \ref{sec_nlin}.

\textbf{Enhanced estimates:}  The next step is to show that, at least in a small energy context, control of the horizontal energy and dissipation actually provides control of non-horizontal derivatives and of the pressure.  More precisely, we aim to prove estimates of the form 
\begin{equation}\label{summary_3}
 \E \ls \bar{\E}  \text{ and } \D \ls \bar{\D} 
\end{equation}
by employing a variety of elliptic estimates and auxiliary estimates.  It is here that the ``non-geometric'' form of \eqref{ns_geometric} becomes particularly useful, as it allows us to apply elliptic theory for the standard constant-coefficient Stokes problem rather than deal with the Stokes problem with coefficients in Sobolev spaces.  It is also worth noting that in this analysis the horizontal derivatives and the temporal derivatives play distinct and important roles.  In particular, the horizontal derivatives are used along with trace theory and Stokes estimates with Dirichlet boundary conditions in a crucial way to decouple certain bulk estimates for $u$ and $p$ from estimates for $\eta$ and $c$.  The details of these enhanced estimates are developed in  Section \ref{sec_aprioris}.

\textbf{Bounds and decay:}  The final step of our nonlinear energy method combines the above to deduce a closed system of a priori estimates that yield both bounds and decay information.  In particular, we use \eqref{summary_1}, \eqref{summary_2}, and \eqref{summary_3} together with the coercivity estimate  $\bar{\E} \le \E \ls \D$ to show that (again in a small energy context)
\begin{equation*}
 \frac{d}{dt} \bar{\E} + \lambda \bar{\E} \le 0 \text{ and } \int_0^T \D(t) dt \ls \E(0)
\end{equation*}
for some universal constant $\lambda$.  Upon integrating the first differential inequality and again appealing to \eqref{summary_3} we find that $\E$ decays exponentially.  The latter inequality tells us that the dissipation is integrable.  We complete this argument and develop  the proofs of Theorems \ref{aprioris_intro} and \ref{gwp_intro} in Section \ref{sec_mains}.

\subsection{Definitions and terminology}\label{def_and_term}

We now mention some of the definitions, bits of notation, and conventions that we will use throughout the paper.

\textbf{  Einstein summation and constants: }  We will employ the Einstein convention of summing over  repeated indices for vector and tensor operations.  Throughout the paper $C>0$ will denote a generic constant that can depend on the parameters of the problem, and $\Omega$, but does not depend on the data, etc.  We refer to such constants as ``universal.''  They are allowed to change from one inequality to the next.  We will employ the notation $a \ls b$ to mean that $a \le C b$ for a universal constant $C>0$.

\textbf{  Norms: } We write $H^k(\Omega)$ with $k\ge 0$ and  $H^s(\Sigma)$ with $s \in \Rn{}$ for the usual Sobolev spaces.  We will typically write $H^0 = L^2$.  To avoid notational clutter, we will avoid writing $H^k(\Omega)$ or $H^k(\Sigma)$ in our norms and typically write only $\norm{\cdot}_{k}$ for $H^k(\Omega)$ norms and $\norm{\cdot}_{\Sigma,s}$ for $H^s(\Sigma)$ norms.

\section{Energy-dissipation equations }\label{sec_ed}

In this section we present two forms of the energy-dissipation equation for solutions to \eqref{ns_geometric}.  The two forms are determined by different ways of linearizing \eqref{ns_geometric}.  The first form is ideal for estimating temporal derivatives, while the second form is ideal for estimating horizontal spatial derivatives and for elliptic regularity.  Finally, we conclude the section with a key lemma for handling nonlinearities.

\subsection{Geometric form }

Here we consider a linear formulation of \eqref{ns_geometric} that is faithful to the geometric significance of the coefficients in \eqref{ns_euler}.  We assume that $u$ and $\eta$ are given and that $\A,\N, J$, etc. are given in terms of $\eta$ as in \eqref{ns_geometric}. We then consider the following system for $(v,q, \zeta,h)$:
\begin{equation}\label{ns_geometric_hybrid}
\begin{cases}
\partial_t  v  - \partial_{t}  \bar{\eta} \tilde{b} K \partial_{3} v +  {u} \cdot \nabla_{\mathcal{A}}   v
+ \text{div}_{\mathcal{A}} S_{\mathcal{A}} (q,v)= F^{1}
 & \text{in }
\Omega  \\
\text{div}_{\mathcal{A}}    {v}=F^{2} & \text{in }\Omega  \\
S_{\mathcal{A}}( {q}, {v}) \mathcal{N} = { \zeta} \mathcal{N} - \sigma_{0} \Delta_{*} \zeta \mathcal{N} - \sigma^{\prime}_{0} \nabla_{ *} h + F^{3}  & \text{on } \Sigma \\
\partial_t  {\zeta}- {v} \cdot \mathcal{N}  = F^{4} &
\text{on } \Sigma \\
\partial_t  {h}  +   c_0 \ \text{div}_{*}  {v}- \gamma \Delta_{ *}  {h}  = F^{5}& \text{on }\Sigma \\
 {u} = 0 & \text{on } \Sigma_{b},
\end{cases}
\end{equation}
where $\sigma_0$ and $\sigma'_0$ are as defined in \eqref{surf_ten_pert}.

We now record the energy-dissipation equality associated to solutions to \eqref{ns_geometric_hybrid}.

\begin{prop}\label{ed_geometric}
Let $u$ and $\eta$ be given and solve \eqref{ns_geometric}.  If $(v,q,\zeta,h)$ solve \eqref{ns_geometric_hybrid} then 
\begin{multline}\label{ed_geometric_0}
 \frac{d}{dt} \left( \int_{\Omega} \frac{|v|^{2}}{2} J +  \int_{\Sigma} \frac{ |\zeta|^{2}}{2} 
+ \int_{\Sigma} \sigma_{0} \frac{|\nabla_{*} \zeta|^{2}}{2} + 
\frac{- \sigma_{0}^{\prime}}{c_0} \int_{\Sigma} \frac{|h|^{2}}{2}
 \right) + \int_{\Omega} \frac{|\mathbb{D}_{\mathcal{A}} v|^{2}}{2} J
 +  \frac{-\gamma  \sigma_{0}^{\prime}}{c_0} \int_{\Sigma} |\nabla_{*} h|^{2} 
\\
=  \int_{\Omega}  ( v\cdot F^{1} + q F^{2})J + \int_{\Sigma} -v\cdot F^{3}   + \int_{\Sigma} (\zeta - \sigma_{0} \Delta_{*} \zeta) F^{4} +  \frac{- \sigma_{0}^{\prime}}{c_0} \int_{\Sigma} h \cdot F^{5}.
\end{multline}
\end{prop}
\begin{proof}
We take the dot product of the first equation \eqref{ns_geometric_hybrid} with $J v$ and integrate over $\Omega$ to find that 
\begin{equation*}
I + II = III, 
\end{equation*}
for 
\begin{eqnarray*}
I &=& \int_{\Omega} \p_{t} v_{i} J v_{i} - \p_{t} \bar{\eta} \tilde{b} \p_{3} v_{i} v_{i} + u_{j} \mathcal{A}_{jk} \p_{k} v_{i} J v_{i},\\
II &=& \int_{\Omega} \mathcal{A}_{jk} \p_{k} S_{ij}(v,q) J v_{i},\ \ \ 
III  \ = \  \int_{\Omega} F^{1} \cdot v J.
\end{eqnarray*}

A simple computation (see for instance Lemma 2.1 of \cite{GT2} for details) shows that 
\begin{equation*}
I = \frac{d}{dt} \int_{\Omega} \frac{|v|^{2}}{2} J. 
\end{equation*}

To handle the term $II$ we first integrate by parts:
\begin{eqnarray*}
II&=& \int_{\Omega} - \mathcal{A}_{jk} S_{ij}(v,q) J \p_{k} v_{i}  + \int_{\Sigma} J \mathcal{A}_{j3} S_{ij}(v,q) v_{i}\\
&=& \int_{\Omega} - q \mathcal{A}_{ik} \p_{k} v_{i} J + J \frac{|\mathbb{D}_{\mathcal{A}} v|^{2}}{2} + \int_{\Sigma} S_{ij}(v,q) \mathcal{N}_{j} v_{i}\\
&=& \int_{\Omega} - q J F^{2} + J \frac{|\mathbb{D}_{\mathcal{A}} v|^{2}}{2}
+ \int_{\Sigma} (\zeta \mathcal{N}- \sigma_{0} \Delta_{*} \zeta \mathcal{N} - \sigma_{0}^{\prime} \nabla_{*} h  ) \cdot v + F^{3} \cdot v.
\end{eqnarray*}
From the forth equation of \eqref{ns_geometric_hybrid} we may compute
\begin{eqnarray*}
 \int_{\Sigma} (\zeta \mathcal{N}- \sigma_{0} \Delta_{*} \zeta \mathcal{N}    ) \cdot v 
&=& \int_{\Sigma} (\zeta  - \sigma_{0} \Delta_{*} \zeta  ) (\p_{t} \zeta - F^{4})\\
&=& \frac{d}{dt} \left( \int_{\Sigma} \frac{|\zeta|^{2}}{2} + \int_{\Sigma} \sigma_{0} \frac{|\nabla_{*} \zeta|^{2}}{2} \right)
- \int_{\Sigma} (\zeta  - \sigma_{0} \Delta_{*} \zeta  ) F_{4}.
\end{eqnarray*}
Now we multiply the fifth equation of \eqref{ns_geometric_hybrid} by $\frac{- \sigma_{0}^{\prime}}{ c_0} h$  and integrate over $\Sigma$ to see that 
\begin{eqnarray*}
\frac{d}{dt} \left( \frac{- \sigma_{0}^{\prime}}{c_0} \int_{\Sigma} \frac{|h|^{2}}{2} \right)+ \sigma_{0}^{\prime}   \int_{\Sigma} v \cdot \nabla_{*} h 
+ \frac{- \gamma \sigma_{0}^{\prime}}{c_0}  \int_{\Sigma} |\nabla_{*} h|^{2} =\frac{- \sigma_{0}^{\prime}}{c_0}  \int_{\Sigma} F^{5}  h.
\end{eqnarray*}
The equality \eqref{ed_geometric_0} then follows by combining the above computations of $I$ and $II$ with the definition of $III$.
\end{proof}

We will employ the form \eqref{ns_geometric_hybrid} to study the temporal derivative of solutions to \eqref{ns_geometric}.  That is, we will apply $\dt$ to \eqref{ns_geometric} to deduce that $(v,q, \zeta,h) = (\p_{t}u, \p_{t} p, \p_{t} \eta, \p_{t} c)$ satisfy \eqref{ns_geometric_hybrid} for certain terms $F^i$.  Below we record the form of these forcing terms $F^i$, $i=1,\dotsc,5$ for this particular problem.  For brevity we will use $\tilde{c}$ in these even though $c$ is the unknown of interest; recall that they are related via \eqref{c_def}.

We have that $F^{1} = \sum_{i=1}^{6}F^{1,i},$ for 
\begin{equation}\label{nlin_F1}
\begin{split}
F^{1,1}_{i} &:= \p_{t} (\p_{t} \bar{\eta} \tilde{b} K) \p_{3} u_{i} ,\\
F^{1,2}_{i} &:= - \p_{t}(u_{j} \mathcal{A}_{jk}) \p_{k} u_{i} + \p_{t}\mathcal{A}_{ik} \p_{k}  p ,\\
F^{1,3}_{i} &:= \p_{t} \mathcal{A}_{j\ell} \p_{\ell} ( \mathcal{A}_{im} \p_{m} u_{j} + \mathcal{A}_{jm} \p_{m} u_{i}),\\
F^{1,4}_{i} &:=  \mathcal{A}_{j\ell} \p_{\ell} \p_{t} ( \mathcal{A}_{im} \p_{m} u_{j} + \mathcal{A}_{jm} \p_{m} u_{i}),\\
F^{1,5}_{i} &:= \p_{t}^{2} \bar{\eta} \tilde{b} K \p_{3} u_{i} ,  \ \ \text{and} \ \ 
F^{1,6}_{i}  :=  \mathcal{A}_{jk} \p_{k} (\p_{t} \mathcal{A}_{i\ell} \p_{\ell} u_{j} + \p_{t} \mathcal{A}_{j\ell} \p_{\ell} u_{i}),
\end{split}
\end{equation}
\begin{equation}\label{nlin_F2}
F^{2}   :=    - \p_{t} \mathcal{A}_{ij} \p_{j} u_{i}, 
\end{equation}
$F^3 = F^{3,1} + F^{3,2} + F^{3,3}$, where for $i=1,2,3$ we have
\begin{equation}\label{nlin_F3}
\begin{split}
F^{3,1}_{i} &:=  (\eta - p) \p_{t} \mathcal{N}_{i}
  +  \left( \mathcal{A}_{ik} \p_{k}u_{j} + \mathcal{A}_{jk} \p_{k} u_{i}  \right) \p_{t} \mathcal{N}_{j}
  +  \left( \p_{t} \mathcal{A}_{ik} \p_{k} u_{j} + \p_{t} \mathcal{A}_{jk} \p_{k} u_{i} \right) \mathcal{N}_{j},
  \\
F^{3,2}_{i} &:= - \sigma^{\prime}(\tilde{c}) \p_{t} c H \mathcal{N}_{i} - ( \sigma(\tilde{c}) - \sigma_{0}) \p_{t} H \mathcal{N}_{i}
- (\sigma_{0} \p_{t} H -\sigma_{0} \p_{t} \Delta_{*} \eta) \mathcal{N}_{i}  
- \sigma(\tilde{c}) H \p_{t} \mathcal{N}_{i},\\
F^{3,3} _{i} &:= - \frac{\nabla_{*} \eta \cdot \nabla_{*} \p_{t} \eta}{
\sqrt{
1+ |\nabla_{*} \eta|^{2}}} \sigma^{\prime}(\tilde{c}) (\nabla_\Gamma \tilde{c})_{i}- \sqrt{1+ |\nabla_{*} \eta|^{2}}\sigma^{\prime\prime}(\tilde{c} ) \p_{t} \tilde{c} (\nabla_\Gamma \tilde{c})_{i},\\
 & +  \sqrt{1+ |\nabla_{*} \eta|^{2}} \sigma^{\prime}(\tilde{c}) \nu_{i} (\nu_{*} \cdot \nabla_{*}) \p_{t}c-\sqrt{1+ |\nabla_{*} \eta|^{2}} \sigma^{\prime}(\tilde{c})
\left\{
\p_{t} \nu_{i} (\nu_{*} \cdot \nabla_{*}) \tilde{c} + \nu_{i} (\p_{t} \nu_{*} \cdot \nabla_{*}) \tilde{c}
\right\}, 
\end{split}
\end{equation}
\begin{equation}\label{nlin_F4}
 F^{4} := \p_{t} D \eta \cdot u,
\end{equation}
and  
\begin{equation}\label{nlin_F5}
 F^5 := \dt \left\{   - u\cdot \nabla_{*} c -  c  \diverge_\Gamma u 
+ \gamma [ \Delta_\Gamma c -\Delta_{*} c]  - c_{0}[ \diverge_\Gamma u - \mathrm{div}_{*} u ] \right\}.
\end{equation}

\subsection{Perturbed linear form  }

Next we consider an alternate way of linearizing \eqref{ns_geometric} that eliminates the $\A$ coefficients in favor or constant coefficients.  This is advantageous for applying elliptic regularity results and is the context in which we will derive estimates for horizontal spatial derivatives.  We may rewrite \eqref{ns_geometric} as
\begin{equation}\label{ns_perturbed}
\begin{cases}
\partial_t u   -\Delta u   + \nabla  p =  G^{1}& \text{in }
\Omega  \\
\text{div} \ {u}=G^{2} & \text{in }\Omega  \\
(pI - \mathbb{D}u - \eta I + \sigma_{0}  \Delta_{*} \eta ) e_{3} + \sigma^{\prime}_{0}  \nabla_{*} c
=  G^{3} & \text{on } \Sigma \\
\partial_t \eta - u_{3}= G^{4}&
\text{on } \Sigma \\
\partial_t {c}   +    c_0   \text{div}_{*}u - \gamma \Delta_{ *}c
= G^{5}  & \text{on }\Sigma \\
u = 0 & \text{on } \Sigma_{b},
\end{cases}
\end{equation}
where $\sigma_0$ and $\sigma'_0$ are as defined in \eqref{surf_ten_pert}.  Here we have written the nonlinear terms $G^i$ for $i=1,\dotsc,5$ as follows.  We write $G^{1 }   := G^{1,1} + G^{1,2} + G^{1,3} + G^{1,4}+ G^{1,5},$ for  
\begin{equation}\label{nlin_G1}
\begin{split}
G^{1,1}_{i} & : = (\delta_{ij} - \mathcal{A}_{ij}) \partial_{j} p  ,\\
G^{1,2}_{i} & : = u_{j} \mathcal{A}_{jk} \partial_{k} u_{i},\\
G^{1,3}_{i} & : =  [K^{2} (1+ A^{2} + B^{2}) -1] \partial_{33} u_{i} 
- 2 AK \partial_{13} u_{i} -2 B K \partial_{23} u_{i},\\
G^{1,4}_{i} & := [-K^{3} (1+ A^{2} + B^{2}) \partial_{3} J
+ AK^{2} (\partial_{1} J + \partial_{3} A) + BK^{2 } (\partial_{2} J + \partial_{3} B) - K (\partial_{1} A + \partial_{2} B)  
]  \partial_{3} u_{i},\\
G^{1,5}_{i} & : = \partial_{t} \bar{\eta} (1+ x_{3}/ b) K \partial_{3} u_{i},
\end{split}
\end{equation}
\begin{equation}\label{nlin_G2}
G^{2} := AK \partial_{3} u_{1} + BK \partial_{3} u_{2} + (1-K) \partial_{3} u_{3}, 
\end{equation}
$G^3 =  G^{3,1}  + G^{3,2} + G^{3,3}+ G^{3,4}$, for 
\begin{equation}\label{nlin_G3}
\begin{split}
G^{3,1} : = &  \partial_{1} \eta 
\begin{pmatrix}
p - \eta - 2 (\partial_{1} u_{1} - A K \partial_{3} u_{1} )\\
- \partial_{2} u_{1} - \partial_{1} u_{2} + BK \partial_{3} u_{1} + AK \partial_{3 } u_{2} \\
- \partial_{1} u_{3} - K \partial_{3} u_{1} + AK \partial_{3} u_{3} 
\end{pmatrix}
\\
  + &  \partial_{2} \eta 
\begin{pmatrix}
  - \partial_{2} u_{1} - \partial_{1} u_{2} + BK \partial_{3} u_{1} + AK \partial_{3} u_{2} \\
  p - \eta -2 ( \partial_{2} u_{2} - BK \partial_{3} u_{2}) \\
  - \partial_{2} u_{3} - K \partial_{3} u_{2} + BK \partial_{3} u_{3} 
\end{pmatrix}
+ 
\begin{pmatrix}
    (K-1) \partial_{3} u_{1} + AK \partial_{3} u_{3} \\
   (K-1) \partial_{3} u_{2} + BK \partial_{3} u_{3} \\
   2(K-1) \partial_{3} u_{3}
\end{pmatrix}
,\\
  G^{3,2} :=& (   \sigma(  c_0 + c) -  \sigma(  c_0 )   ) \Delta_{*} \eta e_{3}
   + \sigma(  c_0 +c) (    H(\eta) - \Delta_{*} \eta)\mathcal{N} 
   + \sigma(  c_0 +c  ) \Delta_{*} \eta (    \mathcal{N} - e_{3}),\\
    G^{3,3} :=&(  \sqrt{1+ |\nabla_{*} \eta|^{2}}-1) \sigma^{\prime} (  c_0 + c ) \nabla_{  *} c + ( \sigma^{\prime} (  c_0  + c) - \sigma_{0}^{\prime} ) \nabla_{  *} c +\sqrt{1+ |\nabla_{*} \eta|^{2}} \sigma^{\prime}(c)( \nabla_\Gamma c -  \nabla_{*} c  ),\\
    G^{3,4} : = & \sigma^{\prime}( c_0 + c  ) \nu_{*} \cdot \nabla_{*} c e_{3},
\end{split}
\end{equation}
\begin{equation}\label{nlin_G4}
G^{4} := -\partial_{1} \eta u_{1} - \partial_{2} \eta u_{2}, 
\end{equation}
and
\begin{equation}\label{nlin_G5}
\begin{split}
G^{5} &: =  - u\cdot \nabla_{*} c -  c \ \diverge_\Gamma {u} 
+ \gamma [ \Delta_\Gamma c -\Delta_{*} c]  -   c_0 [ \diverge_\Gamma  {u} - \mathrm{div}_{*} {u} ].
\end{split}
\end{equation}

Next we consider the energy-dissipation evolution equation for solutions to problems of the form \eqref{ns_perturbed}.

\begin{prop}\label{ed_pert} Suppose $(v,q,\zeta,h)$ solve
\begin{equation}\label{eqtn_v}
\begin{cases}
\partial_t v  -\Delta  v   + \nabla  q =  \Phi^{1}& \text{in }
\Omega  \\
\mathrm{div}  \ {v}=\Phi^{2} & \text{in }\Omega  \\
(qI - \mathbb{D}v - \zeta I + \sigma_{0} \Delta \zeta ) e_{3} + \sigma^{\prime}_{0}  \nabla_{*} h=  \Phi^{3} & \text{on } \Sigma \\
\partial_t \zeta - v_{3}= \Phi^{4}&
\text{on } \Sigma \\
\partial_t {h}   +   c_0  \mathrm{div}_{*}v -\gamma \Delta_{*}h= \Phi^{5}  & \text{on }\Sigma \\
v = 0 & \text{on } \Sigma_{b}.%
\end{cases}
\end{equation}
Then 
\begin{multline}\label{ed_pert_0} 
\frac{d}{dt} \left(   \int_{\Omega} \frac{|v|^{2}}{2} + \int_{\Sigma} \frac{|\zeta|^{2}}{2} + \frac{\sigma_{0}}{2} |\nabla_{*} \zeta|^{2}+\frac{-\sigma^{\prime}_{0} }{ 2  c_0} |h|^{2}   \right) 
  + \int_{\Omega} \frac{| \mathbb{D} v| ^{2} }{2}  + \int_{\Sigma} \frac{-\gamma \sigma^{\prime}_{0}}{   c_0} |\nabla_{*} h|^{2}
  \\
= \int_{\Omega} v\cdot \Phi^{1} + q \Phi^{2} - v\cdot \nabla \Phi^{2} + \int_{\Sigma} - v \cdot \Phi^{3} + \zeta \Phi^{4} - \sigma_{0} \Phi^{4} \Delta_{*} \zeta- \frac{\sigma^{\prime}_{0}}{ c_0} \Phi^{5} h .
\end{multline}
\end{prop}
\begin{proof}
From the first equation in \eqref{eqtn_v} we compute
\begin{equation*} 
\begin{split}
\partial_{t} v_{i} + (\mathrm{div} S(q,v))_{i}  = \partial_{t} v_{i} + \partial_{i} q - \Delta v_{i} - \partial_{i} \Phi^{2}  
 = \Phi_{i}^{1} - \partial_{i} \Phi^{2}.
\end{split}
\end{equation*}
By the usual energy estimates (see for instance Lemma 2.3 in \cite{GT2}) we may compute
\begin{equation*}
\frac{d}{dt} \int_{\Omega} \frac{|v|^{2}}{2} 
+ \int_{\Omega} \frac{|\mathbb{D}u|^{2}}{2}
+ \underbrace{\int_{\Sigma} v_{3} \zeta}_{I} + 
\underbrace{   \int_{\Sigma} - \sigma_{0} v_{3} \Delta_{*} \zeta }_{II}  + \underbrace{ \int_{\Sigma} - \sigma^{\prime}_{0}  v\cdot \nabla_{*} h }_{III}
= \int_{\Omega} v\cdot \Phi^{1}  +q \Phi^{2} - v\cdot \nabla \Phi^{3}. 
\end{equation*}

We compute $I$ by integrating by parts and using \eqref{eqtn_v}:
\begin{equation*}
I = \int_{\Sigma}  \zeta \partial_{t} \zeta  - \zeta \Phi^{4}  =
\frac{d}{dt} \int_{\Sigma} \frac{|\zeta|^{2}}{2} - \int_{\Sigma} \zeta \Phi^{4}.
\end{equation*}
Similarly, 
\begin{equation*}
II =  - \sigma_{0} \int_{\Sigma} \{ \partial_{t} \zeta - \Phi^{4}  \} \Delta_{*} \zeta 
=  \sigma_{0} \int_{\Sigma} \partial_{t} \nabla_{*} \zeta \cdot \nabla_{*} \zeta + \sigma_{0} \Phi^{4} \Delta_{*} \zeta 
= \sigma_{0}  \frac{d}{dt} \int_{\Sigma} \frac{|\nabla_{*} \zeta|^{2}}{2}
+ \sigma_{0} \int_{\Sigma} \Phi^{4 } \Delta_{*} \zeta.
\end{equation*}
Finally, for $III$ we compute
\begin{equation*}
\begin{split}
III = & \int_{\Sigma} \sigma^{\prime}_{0}  h \ \mathrm{div}_{*} v 
 =  \int_{\Sigma} \sigma^{\prime} _{0} \frac{h}{ c_0} \{  - \partial_{t}h + \gamma \Delta_{*} h + \Phi^{5} \}\\
=& \frac{d}{dt} \int_{\Sigma} \frac{-\sigma^{\prime} _{0}}{2 c_0}|h|^{2}
+ \int_{\Sigma} \frac{-\gamma \sigma^{\prime}_{0}}{ c_0} |\nabla_{*} h|^{2}
+ \int_{\Sigma} \frac{\sigma^{\prime}_{0}}{ c_0}   \Phi^{5}h.
\end{split}
\end{equation*}

The equality \eqref{ed_pert_0} then follows by combining the above computations.
\end{proof}

\section{Estimates of the nonlinearities  }\label{sec_nlin}

In this section we record estimates for the nonlinearities that appear in \eqref{ns_geometric_hybrid} and \eqref{ns_perturbed}.  Throughout this section we will repeatedly use the estimates of Lemmas \ref{p_poisson} and \ref{p_poisson_2} to estimate $\bar{\eta}$, as well as Lemma  \ref{product_ests} to estimate various nonlinearities.  For the sake of brevity we will use these lemmas without explicit reference.

\subsection{Useful $L^\infty$ estimates  }

We begin the section by recording the following result, which is useful for removing the appearance of $J$ and $\mathcal{A}$ factors.

\begin{lem}\label{infinity_bounds}
There exists a universal $0 < \delta < 1$ so that if $\ns{\eta}_{5/2} \le \delta$, then the following hold.
\begin{enumerate}
 \item We have the estimate
 \begin{equation*}\label{infb_01}
 \pnorm{J-1}{\infty}^2  +\pnorm{A}{\infty}^2 + \pnorm{B}{\infty}^2 \le \hal, \text{ and } 
  \pnorm{K}{\infty}^2+ \pnorm{\mathcal{A}}{\infty}^2 \ls 1.
\end{equation*}

 \item The map $\Theta$ defined by \eqref{mapping_def} is a diffeomorphism.
 \item There exists a universal constant $C>0$ such that for all $v \in H^1(\Omega)$ such that $v =0$ on $\Sigma_b$ we have that
\begin{equation*}\label{infb_02}
  \int_\Omega \abs{\sg v}^2 \le \int_\Omega J \abs{\sg_{\mathcal{A}} v}^2 + C \sqrt{\E}\int_\Omega \abs{\sg v}^2 . 
\end{equation*}
 \item We have the estimates
\begin{equation*}
 -\frac{c_0}{2} \le c \le \frac{c_0}{2} \text{ and }  \frac{c_0}{2} \le \tilde{c} \le \frac{3 c_0}{2}
\end{equation*}

\end{enumerate}

\end{lem}
\begin{proof}
See Lemma 2.4 in \cite{GT2} for a proof of the first two items.  The proof of the third item can be found in the proof of Proposition 4.3 in \cite{GT2}.  The third item follows directly from the Sobolev embedding and the identity \eqref{c_def}.
\end{proof}

\subsection{Nonlinearities in \eqref{ns_geometric_hybrid}}

Our goal now is to estimate the nonlinear terms $F^i$ for $i=1,\dotsc,5$, as defined in \eqref{nlin_F1}--\eqref{nlin_F5}.  These estimates will be used principally to estimate the interaction terms on the right side of \eqref{ed_geometric_0}.

\begin{thm}\label{nlin_ests_F}
Let $F^1,\dotsc,F^5$ be as defined in  \eqref{nlin_F1}--\eqref{nlin_F5}.  Let $\E$ and $\D$ be as defined in \eqref{def_E} and \eqref{def_D}. Suppose that $\E \le \delta$, where $\delta \in (0,1)$ is the universal constant given in Lemma \ref{infinity_bounds}, and that $\D < \infty$.   Then  
\begin{equation}\label{nlin_ests_F_01}
\norm{F^{1} J}_{0}  +  \norm{F^{3}}_{\Sigma,0} +  \norm{F^{4 } }_{\Sigma,0}  \lesssim \sqrt{\mathcal{E}}\sqrt{\mathcal{D}},
\end{equation}
\begin{equation}\label{nlin_ests_F_02} 
\abs{\int_\Omega \p_{t} p  F^{2} J  - \frac{d}{dt} \int_\Omega p F^2 J } \ls \sqrt{\E}\D \text{ and } \abs{\int_\Omega p F^2 J } \ls \E^{3/2},
\end{equation}
and
\begin{equation}\label{nlin_ests_F_03}
\abs{ \int_\Omega \dt c F^5 } \ls \sqrt{\E} \D.
\end{equation}
\end{thm}
\begin{proof}
We divide the proof into several steps.  Throughout the lemma we will employ H\"older's inequality, Sobolev embeddings, trace theory, and Lemma \ref{infinity_bounds}.

\emph{Step 1:}   $F^{1},F^{3,1},$ and $F^{4}$ estimates

The estimate 
\begin{equation}\label{nlin_ests_F_1}
\norm{F^{1} J}_{0}  +  \norm{F^{3,1}}_{\Sigma,0} +  \norm{F^{4 } }_{\Sigma,0}  \lesssim \sqrt{\mathcal{E}}\sqrt{\mathcal{D}}
\end{equation}
is proved in  \cite{WTK}.

\emph{Step 2:} $F^{3,2}$ estimate

We bound the first term in $F^{3,2}$ via
\begin{eqnarray*}
 \| \sigma^{\prime}(\tilde{c}) \p_{t} c H \mathcal{N} \|_{\Sigma,0} 
 &\lesssim&
 \| \p_{t} c \|_{L^{4}(\Sigma)} \| \nabla_{*} ^{2} \eta \|_{L^{4}(\Sigma)} (1+ \| \nabla_{*} \eta \|_{L^{\infty}(\Sigma)})^{2} \\
 &\lesssim&   \| \p_{t}c \|_{\Sigma, 1} \| \eta \|_{\Sigma, 3} (1+ \|  \eta \|_{\Sigma, 2+})^{2} \\
 &\lesssim& \sqrt{\mathcal{D}}\sqrt{\mathcal{E}}(1+ \mathcal{E}).
\end{eqnarray*}

To bound the second term we first write 
\begin{equation*}
 H(\eta) = \Delta_\ast \eta +  \Delta_\ast \eta \left(\frac{1}{\sqrt{1+\abs{\nab_\ast \eta}^2}} -1\right) - \frac{\p_i \p_j \eta \p_i \eta \p_j \eta}{(1+\abs{\nab_\ast \eta}^2)^{3/2}}
\end{equation*}
in order to bound
\begin{eqnarray*}
&&\norm{ \dt (H(\eta) - \Delta_\ast \eta ) }_{L^{2}(\Sigma)}\\
&\lesssim& \big\{ \| \p_{t} \nabla_{*}^{2} \eta \|_{L^{2}(\Sigma)} \| \nabla_{*} \eta \|_{L^{\infty}(\Sigma)} 
+  \| \p_{t} \nabla_{*} \eta \|_{L^{\infty}(\Sigma)} \| \nabla_{*}^{2} \eta \|_{L^{2}(\Sigma)}
\big\}
\left( 1+ \norm{  \frac{\nabla_{*} \eta }{ [1+ |\nabla_{*} \eta|^{2}]^{3/2}} }_{L^{\infty}(\Sigma)} \right)\\
&\lesssim& \{ \sqrt{\mathcal{D}} \sqrt{\mathcal{E}} + \sqrt{\mathcal{D}} \sqrt{\mathcal{E}}    \} (1+ \sqrt{\mathcal{E}}).
\end{eqnarray*}
Next we employ the simple identity 
\begin{equation*}
 \sigma(\tilde{c}) - \sigma_{0} = \int^{\tilde{c}}_{c_0} \sigma^{\prime}(s) \dd s
\end{equation*}
to estimate
\begin{equation*}
 \norm{\sigma(\tilde{c}) - \sigma_{0}}_{L^\infty(\Sigma)} \ls \norm{\sigma}_{C^1} \norm{c_0}_{L^\infty(\Sigma)}.
\end{equation*}
Combining these, we deduce that 
\begin{eqnarray*}
&&\| (\sigma(\tilde{c}) - \sigma_{0} ) \p_{t} H \mathcal{N} \|_{\Sigma,0} \\
&\lesssim& \| \sigma(\tilde{c}) - \sigma_{0} \|_{L^{\infty}(\Sigma)}  \| \p_{t} H \|_{L^{2}(\Sigma)} \| \mathcal{N} \|_{L^{\infty}(\Sigma)}\\
&\lesssim&  \| \sigma \|_{C^{1}} \| c_0 \|_{\Sigma, 2}  [\norm{\dt \eta}_{\Sigma,2}+  \sqrt{\E \D}(1+\sqrt{\E})]
\\
&\lesssim& \sqrt{\mathcal{E}}\sqrt{\mathcal{D}}.
\end{eqnarray*}

Similarly, we bound the third term in $F^{3,2}$ via
 \begin{eqnarray*}
&& \| \sigma_{0} (\p_{t} H - \p_{t} \Delta_{*} \eta ) \mathcal{N} \|_{\Sigma, 0}\\
&\lesssim& \| \p_{t} H - \p_{t} \Delta_{*} \eta  \|_{L^{2}(\Sigma)} \| \mathcal{N } \|_{L^{\infty}(\Sigma)}\\
 &\lesssim& \{ \sqrt{\mathcal{D}} \sqrt{\mathcal{E}} + \sqrt{\mathcal{D}} \sqrt{\mathcal{E}}    \} (1+ \sqrt{\mathcal{E}}) \sqrt{\mathcal{E}},
 \end{eqnarray*}
and we bound the fourth by
\begin{eqnarray*}
 &&\| \sigma(\tilde{c}) H \p_{t} \mathcal{N} \|_{L^{2}(\Sigma)}\\
 &\lesssim& \| \sigma \|_{L^{\infty}(\Sigma)} \|\p_{t}  \nabla_{*} \eta\|_{L^{\infty}(\Sigma)} \| H \|_{L^{2}(\Sigma)}  \\
 &\lesssim& \| \p_{t} \eta \|_{ H^{2+}(\Sigma) }  \| \eta \|_{H^{2}(\Sigma)}(1+ \| \nabla_{*} \eta \|_{L^{\infty}(\Sigma)})^{2} \\
 &\lesssim& \sqrt{\mathcal{D}} \sqrt{\mathcal{E}} (1+ \sqrt{\mathcal{E}})^{2}.
\end{eqnarray*}

Combining the above estimates and again using the fact that $\E \le 1$, we deduce that 
\begin{equation}\label{nlin_ests_F_2}
 \norm{F^{3,2}}_{\Sigma,0} \ls \sqrt{\E} \sqrt{\D}.
\end{equation}

\emph{Step 3:} $F^{3,3}$ estimate

According to the usual Sobolev embedding $H^{1+}(\Sigma) \hookrightarrow L^{\infty}(\Sigma)$, we may estimate
\begin{eqnarray*}
&&\| F^{3,3} \|_{\Sigma, 0}\\
 &\lesssim& \| \nabla_{*} \eta \|_{L^{\infty}(\Sigma)} \| \nabla_{*} \p_{t} \eta \|_{L^{2}(\Sigma)} \| \nabla_{*} ( c_0 ) \|_{L^{\infty}(\Sigma)}\\
&& + (1 + \| \nabla_{*} \eta \|_{L^{\infty}(\Sigma)})  \| \p_{t} \tilde{c} \|_{L^{2}(\Sigma)} \| \nabla_{*} ( c_0 ) \|_{L^{\infty}(\Sigma)}\\
&&+ (1+ \| \nabla_{*} \eta \|_{L^{\infty}(\Sigma)})^{3}
\| \sigma \|_{C^{1}} \| \nabla_{*} \p_{t} \tilde{c} \|_{L^{2}(\Sigma)}\\
&&+
(1+ \| \nabla_{*} \eta \|_{L^{\infty}(\Sigma)})^{2} \| \sigma \|_{C^{1}}
\| \p_{t} \nabla_{*} \eta \|_{L^{\infty}(\Sigma)} \| \nabla_{*} (\tilde{c}
- c_0 )\|_{L^{2}(\Sigma)}
\\
&\lesssim& (1+ \sqrt{\mathcal{E}}) \sqrt{\mathcal{E}}\sqrt{\mathcal{D}}.
\end{eqnarray*}
Again since $\E \le 1$ we find that 
\begin{equation}\label{nlin_ests_F_3}
  \norm{F^{3,3}}_{\Sigma,0} \ls \sqrt{\E} \sqrt{\D}.
\end{equation}
Then by combining \eqref{nlin_ests_F_1}, \eqref{nlin_ests_F_2}, and \eqref{nlin_ests_F_3} we deduce that \eqref{nlin_ests_F_01} holds.

\emph{Step 4:} $F^2$ estimate

We have that 
\begin{equation*}
\int_\Omega \dt p J F^2 = \frac{d}{dt} \int_{\Omega} p J F^2 - \int_\Omega p (\dt J F^2 + J \dt F^2).
\end{equation*}
We may then use the definition of $F^2$ and $J$ to  estimate 
\begin{eqnarray*}
&& \abs{\int_\Omega p (\dt J F^2 + J \dt F^2)  }\\
&\lesssim&   \| p \p_{t}^{2} \nabla_{*} \bar{\eta}  \nabla u (1 + |\bar{\eta}| + |\nabla \bar{\eta}|) \|_{L^{1}(\Omega)}
+  \|p \p_{t} \nabla_{*} \bar{\eta}  \nabla u (|\p_{t} \bar{\eta}| + |\p_{t} \nabla\bar{\eta}|)   \|_{L^{1}(\Omega)}\\
&&+  \|  p   \nabla \bar{\eta} \p_{t} \nabla u (|\p_{t} \bar{\eta}| + | \p_{t} \nabla \bar{\eta}|) \|_{L^{1}(\Omega)}\\
&\lesssim&  \| p \|_{L^{\infty}(\Omega)}\| \p_{t}^{2} \nabla \bar{\eta} \|_{L^{2}(\Omega)} \| \nabla u \|_{L^{4}(\Omega)} \| \nabla \bar{\eta} \|_{L^{4}(\Omega)}\\
&&+  \| p \|_{L^{\infty}(\Omega)} \| \p_{t}  \bar{\eta} \|_{ H^{1}(\Omega)} \big\{ \| \p_{t}  \nabla \bar{\eta} \|_{ L^{4}(\Omega)} \| \nabla u \|_{L^{4}(\Omega)} 
+ \| \nabla \bar{\eta} \|_{L^{\infty}(\Omega)} \|  \p_{t}\nabla   u  \|_{L^{2}(\Omega)}
 \big\}
\\
&\lesssim&
 \| p \|_{2} \| u \|_{2} \big\{ \|
\p_{t}^{2} {\eta} \|_{ \Sigma, \frac{1}{2}} \| \eta \|_{\Sigma, \frac{3}{2}}
+ \| \p_{t}  {\eta} \|_{\Sigma, \frac{1}{2}} \| \p_{t}  {\eta} \|_{\Sigma, \frac{3}{2}} 
\big\}
+   \| p \|_{2} \| \p_{t } \eta \|_{\Sigma, \frac{1}{2}}\| \eta \|_{\Sigma,3}
\| \p_{t} u \|_{1}
\\
&\lesssim&   \sqrt{\mathcal{E}}   \sqrt{\mathcal{E}} {\mathcal{D}},
\end{eqnarray*}
where we have used the embeddings $H^{1}(\Omega) \hookrightarrow L^{4}(\Omega)$ and $H^{3/2+}(\Omega)\hookrightarrow L^{\infty}(\Omega)$.

Similarly,
\begin{eqnarray*}
&&\abs{\int_\Omega p J F^2 }\\
&\lesssim& \| p \|_{L^{4}(\Omega)} \|F^{2} \|_{L^{4/3}(\Omega)} \| J \|_{L^{\infty}(\Omega)}\\
&\lesssim&  \|p \|_{H^{1}(\Omega) }   
\| \nabla u \|_{L^{4}(\Omega)}
 \|  (1+ |\nabla \bar{\eta} |)
 ( |\p_{t}  \bar{\eta}|+ |\p_{t} \nabla \bar{\eta}| ) 
  \|_{L^{2}(\Omega)} 
\| \bar{\eta} \|_{H^{3/2+}(\Omega)}\\
&\lesssim&  \|p \|_{1}
\|  u \|_{2}
(1+\| \eta \|_{\Sigma, 3} )\| \p_{t} \eta \|_{\Sigma,1 }
(1+\|  {\eta} \|_{\Sigma,3})\\
&\lesssim& \mathcal{E}^{3/2} .
\end{eqnarray*} 

By combining the above estimates we then deduce that \eqref{nlin_ests_F_02} holds.

\emph{Step 5:} $F^5$ estimate

To begin we note that 
\begin{eqnarray*}
F^{5} & : =&
- \p_{t}u\cdot \nabla_{*} c - u\cdot \nabla_{*} \p_{t}c
- c \diverge_\Gamma \p_{t} u 
+ [     c \diverge_\Gamma \p_{t} u  - \p_{t} (c \diverge_\Gamma u)]\\
&&+ 
 \gamma \p_{t} [ \Delta_{\Gamma }   c -\Delta_{*} c] 
+ c_{0} \p_{t} [ \diverge_\Gamma  {u} - \mathrm{div}_{*}  {u} ] .
\end{eqnarray*}
We will handle each of these terms in turn.

For the first two we use  trace theory and the Sobolev embedding to estimate
\begin{eqnarray*}
&&\| \p_{t} u \cdot \nabla_{*} c\|_{\Sigma, 0} + \| u \cdot \nabla_{*} \p_{t} c \|_{\Sigma, 0}\\
&\lesssim& \| \p_{t} u \|_{L^{4}(\Sigma)}\|  \nabla_{*} c \|_{L^{4}(\Sigma)}
+ \| u \|_{L^{\infty}(\Sigma)} \| \nabla_{*} \p_{t} c \|_{L^{2}(\Sigma) }\\
&\lesssim& \| \p_{t} u \|_{\Sigma, \frac{1}{2}} \| c \|_{\Sigma, \frac{3}{2}} + \| u \|_{ \frac{3}{2}+} \| \p_{t} c\|_{\Sigma, 1}\\
&\lesssim& \sqrt{\mathcal{E}} \sqrt{\mathcal{D}}.
\end{eqnarray*}
Thus 
\begin{equation*}
\abs{ (- \p_{t}u\cdot \nabla_{*} c - u\cdot \nabla_{*} \p_{t}c, \dt c)_{H^0(\Sigma)} } \ls \sqrt{\E} \sqrt{\D} \norm{\dt c}_{\Sigma,0} \ls \sqrt{\E} \D.
\end{equation*}

For the third term we integrate by part to see that
\begin{eqnarray*}
&&\big(-c \diverge_\Gamma \p_{t} u , \p_{t} c  \big)_{H^{0}(\Sigma)}\\
&=&  \big( \p_{t} u_{i}, \p_{i}[ c \p_{t} c]\big)_{H^{0}(\Sigma)}
+ \big( \p_{t} u_{i}, \p_{j} [ c \p_{t} c \nu_{i} \nu_{j}] \big)_{H^{0}(\Sigma)}
 + \big( \p_{t} u_{3}, \p_{i} [ \nu_{i} \nu_{3} c \p_{t} c]  \big)_{H^{0}(\Sigma)}.
\end{eqnarray*}
Therefore
\begin{eqnarray*}
&&\abs{ \big( - c \diverge_\Gamma \p_{t} u , \p_{t} c  \big)_{H^{0}(\Sigma)}   }  \\
&\lesssim& \|\p_{t} u \|_{\Sigma, 0} \| \nabla_{*} c\|_{\infty} \| \p_{t} c \|_{\Sigma, 0}
+  \|\p_{t} u \|_{\Sigma, 0}  \|  c\|_{\infty} \| \nabla_{*}\p_{t} c \|_{\Sigma, 0} + \| \p_{t} u  \|_{\Sigma,0} \| c \|_{\infty} \| \p_{t} c \|_{\Sigma,0} \| \nabla_{*}^{2} \eta \|_{\infty}\\
&\lesssim&
\| \p_{t} u \|_{\frac{1}{2}+} \| c \|_{\Sigma, \frac{5}{2}} \| \p_{t }c \|_{\Sigma, 0}
+ \| \p_{t} u \|_{\frac{1}{2}+ } \|  c\|_{\Sigma, \frac{3}{2}} \| \p_{t}c \|_{\Sigma, 1}
+ \| \p_{t} u \|_{\frac{1}{2}+} \| c \|_{2} \| \p_{t} c \|_{\Sigma, 0} \| \eta \|_{\Sigma, \frac{7}{2}}
\\
&\lesssim& \sqrt{\mathcal{E}} \mathcal{D}.
\end{eqnarray*}

For the fourth term we have 
\begin{eqnarray*}
&&\|  c \diverge_\Gamma \p_{t} u  - \p_{t} (c \diverge_\Gamma u) \|_{\Sigma,0}=\| \p_{t}c \diverge_\Gamma u + \p_{t} \nu \cdot (\nu_{*} \cdot \nabla_{*}) u  + \nu \cdot (\p_{t}\nu_{*} \cdot \nabla_{*}) u    \|_{\Sigma, 0}\\
&\lesssim& \| \p_{t} c \|_{L^{4}(\Sigma)} \| \nabla_{*} u \|_{L^{4}(\Sigma)}
+ \|  \p_{t} \nabla_{*} \eta \|_{L^{4}(\Sigma)} \| \nabla_{*} u \|_{L^{4}(\Sigma)} \\
&\lesssim& \{ \| \p_{t} c \|_{\Sigma,1} + \| \p_{t} \eta \|_{\Sigma, \frac{3}{2}} \} \| u \|_{2}  \lesssim \sqrt{\mathcal{E}} \sqrt{\mathcal{D}}.
\end{eqnarray*}
Hence
\begin{equation*}
\abs{ \big( c \diverge_\Gamma \p_{t} u  - \p_{t} (c \diverge_\Gamma u) , \dt c \big)_{H^0(\Sigma)} } \ls \sqrt{\E} \sqrt{\D} \norm{\dt c}_{\Sigma,0} \ls \sqrt{\E} \D.
\end{equation*}

Now we consider the fifth term.  Direct computation reveals that
\begin{eqnarray*}
\Delta_\Gamma c = \Delta_{*}c - \nu_{3}  \nu_{*} \cdot \nabla_{*} \p_{j} \eta (\delta_{ij} - \nu_{i} \nu_{j}) \p_{i} c 
 -  [\nu_{j}\p_{j} \nu_{i}  + \nu_{i} \p_{j} \nu_{j} ] \p_{i} c.
\end{eqnarray*}
Therefore 
\begin{eqnarray*}
&&\| \gamma \partial_{t}[ \Delta_\Gamma c - \Delta_{*} c] \|_{\Sigma,0}\\
&\lesssim&     \| \nabla_{*}^{2} \partial_{t} \eta  \nabla_{*}c \|_{\Sigma,0} +      \|  \nabla_{*} \partial_{t} \eta \nabla_{*}^{2} \eta   \nabla_{*}c \|_{\Sigma,0} + 
 \|  \nabla_{*}^{2}\eta  \nabla_{*} \partial_{t} c \|_{\Sigma, 0}\\
  &\lesssim&  \| \nabla_{*}^{2} \partial_{t} \eta\|_{L^{4}(\Sigma)} \| \nabla_{*}c\|_{L^{4}(\Sigma)} \|\partial_{t}c\|_{L^{2}(\Sigma)}
  + \| \nabla_{*} \p_{t} \eta \|_{\infty} \| \nabla_{*} ^{2} \eta \|_{L^{2}(\Sigma) } \| \nabla_{*} c\|_{\infty} \| \p_{t} c \|_{L^{2}(\Sigma)}\\
  &&+ \| \nabla_{*}^{2} \eta \|_{\infty} \| \nabla_{*} \p_{t} c \|_{L^{2}(\Sigma)}
  \| \p_{t} c \|_{L^{2}(\Sigma)} 
  \\
 &\lesssim& \| \p_{t} \eta \|_{\Sigma, \frac{5}{2}} \| c \|_{\Sigma, \frac{3}{2}} \| \p_{t} c \|_{\Sigma, 0}
 + \| \p_{t} \eta \|_{\Sigma, \frac{5}{2}} \| \eta \|_{\Sigma,2} \| c \|_{\Sigma, \frac{5}{2}} \| \p_{t} c \|_{\Sigma, 0}
 + \| \eta \|_{\Sigma, \frac{7}{2}} \| \p_{t} c \|_{\Sigma, 1} \| \p_{t} c \|_{\Sigma, 0}\\
 &\lesssim& \sqrt{\mathcal{E}} \mathcal{D},
\end{eqnarray*}
where we have used the Holder inequality $( \frac{1}{4} + \frac{1}{4} + \frac{1}{2} =1)$ and the Sobolev embeddings $H^{\frac{1}{2}}(\Sigma)\hookrightarrow  L^{4}(\Sigma )$ and $H^{1+}(\Sigma) \hookrightarrow L^{\infty}(\Sigma ) $. Hence
\begin{equation*}
 \abs{ (\gamma\partial_{t}( \Delta_\Gamma c - \Delta_{*} c ), \partial_{t}c  )_{H^{0}(\Sigma)} }  \ls \sqrt{\E} \D \norm{\dt c}_{\Sigma,0} \ls \E \D \ls \sqrt{\E} \D.
\end{equation*}

For the sixth term we note that 
\begin{equation*}
\diverge_\Gamma u - \text{div}_{*} u = - \nu_{i} (\nu_{*} \cdot \nabla_{*})u_{i}. 
\end{equation*}
Therefore
\begin{eqnarray*}
&&\big(c_{0} \p_{t} [ \diverge_\Gamma  {u} - \mathrm{div}_{*}  {u} ] , \p_{t} c\big)_{H^{0}(\Sigma)}\\
&=& -  \big( \nu_{i} ( \nu_{*} \cdot \nabla_{*}) \p_{t} u_{i}, \p_{t}c\big)_{H^{0}(\Sigma)} -  \big(   \p_{t}\nu_{i} ( \nu_{*} \cdot \nabla_{*}) u_{i} +   \nu_{i} (  \p_{t}\nu_{*} \cdot \nabla_{*}) u_{i}, \p_{t}c\big)_{H^{0}(\Sigma)}.  
\end{eqnarray*}
Clearly 
\begin{eqnarray*}
 &&\abs{ \big(   \p_{t}\nu_{i} ( \nu_{*} \cdot \nabla_{*}) u_{i} +   \nu_{i} (  \p_{t}\nu_{*} \cdot \nabla_{*}) u_{i}, \p_{t}c\big)_{H^{0}(\Sigma)}} \\
 &\ls& \| \p_{t} \nabla_{*} \eta \|_{L^{4}(\Sigma)} \| \nabla_{*} u\|_{L^{2}(\Sigma)} \| \p_{t} c\|_{L^{4}(\Sigma)} \lesssim \| \p_{t} \eta \|_{\Sigma, 2} \| u \|_{2} \| \p_{t } c\|_{\Sigma, 1} \lesssim \sqrt{\mathcal{E} } \mathcal{D}.
\end{eqnarray*}
On the other hand, we may integrate by parts to obtain
\begin{eqnarray*}
 \abs{-  \big( \nu_{i} ( \nu_{*} \cdot \nabla_{*}) \p_{t} u_{i}, \p_{t}c\big)_{H^{0}(\Sigma)}  }
 &=& \abs{ \big( \p_{j}\nu_{i}  \nu_{j}  \p_{t} u_{i}, \p_{t}c\big)_{H^{0}(\Sigma)} + \big( \nu_{i} \p_{j} \nu_{j}  \p_{t} u_{i}, \p_{t}c\big)_{H^{0}(\Sigma)} }\\
 &\lesssim& \| \nabla_{*}^{2} \eta \|_{L^{\infty}(\Sigma)} \| \p_{t} u \|_{L^{2}(\Sigma)}  \| \p_{t} c \|_{L^{2}(\Sigma)}  \lesssim \| \eta \|_{\Sigma, \frac{7}{2}}
 \| \p_{t} u \|_{1} \| \p_{t} c\|_{\Sigma,0}\\
 &\lesssim& \sqrt{\mathcal{E}}\mathcal{D}.
\end{eqnarray*}
Thus
\begin{equation*}
 \abs{ \big(c_{0} \p_{t} [ \diverge_\Gamma  {u} - \mathrm{div}_{*}  {u} ] , \p_{t} c\big)_{H^{0}(\Sigma)} } \ls \sqrt{\E} \D.
\end{equation*}

We have now estimated all of the terms appearing in $F^5$, and we conclude that \eqref{nlin_ests_F_03} holds.

\end{proof}

\subsection{Nonlinearities in \eqref{ns_perturbed}}

Now we turn our attention to the nonlinear terms $G^i$ for $i=1,\dotsc,5$, as defined in \eqref{nlin_G1}--\eqref{nlin_G5}.

\begin{thm}\label{nlin_ests_G}
Let $G^1,\dotsc,G^5$ be as defined in  \eqref{nlin_G1}--\eqref{nlin_G5}.  Let $\E$ and $\D$ be as defined in \eqref{def_E} and \eqref{def_D}. Suppose that $\E \le \delta$, where $\delta \in (0,1)$ is the universal constant given in Lemma \ref{infinity_bounds}, and that $\D < \infty$.   Then  
\begin{equation}\label{nlin_ests_G_01}
\| G^{1} \|_{1}  + \| G^{2} \|_{2} + \| G^{3} \|_{\Sigma, \frac{3}{2}} + \| G^{4} \|_{\Sigma, \frac{5}{2}} + \| \dt G^{4} \|_{\Sigma, \frac{1}{2}} + \| G^{5} \|_{\Sigma, 1}   \lesssim  \sqrt{\mathcal{E}} \sqrt{\mathcal{D}},  
\end{equation}
and
\begin{equation}\label{nlin_ests_G_02}
\| G^1 \|_{0}+ \| G^{2} \|_{1} +  \| G^{3} \|_{\Sigma, \frac{1}{2}} + \| G^{4} \|_{\Sigma, \frac{3}{2}} + 
\lesssim \E.
\end{equation}
\end{thm}
\begin{proof}
We again divide the proof into several steps.  Throughout the lemma we will employ H\"older's inequality, Sobolev embeddings, trace theory, and Lemma \ref{infinity_bounds}.

\emph{Step 1:}  $G^1, G^2, G^{3,1},$ and $G^4$ estimates 
 
The estimates
\begin{equation*}
 \norm{G^1}_1 + \norm{G^2}_2 + \norm{G^{3,1}}_{\Sigma,3/2} + \norm{G^4}_{\Sigma,5/2} + \| \dt G^{4} \|_{\Sigma, \frac{1}{2}} \ls \sqrt{\E} \sqrt{\D}
\end{equation*}
and 
\begin{equation*}
\norm{G^1}_0 + \norm{G^2}_1 + \norm{G^{3,1} }_{\Sigma,1/2} + \norm{G^4}_{\Sigma,3/2} \ls \E 
\end{equation*}
are proved in \cite{WTK}.  Thus in order to prove \eqref{nlin_ests_G_01} and \eqref{nlin_ests_G_02} it suffices to prove 
\begin{equation}\label{nlin_ests_G_1}
 \norm{G^{3,2}}_{\Sigma,3/2} + \norm{G^{3,3}}_{\Sigma,3/2} + \norm{G^{3,4}}_{\Sigma,3/2}+ \norm{G^5}_{\Sigma,1} \ls \sqrt{\E} \sqrt{\D}  
\end{equation}
and 
\begin{equation}\label{nlin_ests_G_2}
 \norm{G^{3,2}}_{\Sigma,1/2} + \norm{G^{3,3}}_{\Sigma,1/2} + \norm{G^{3,4}}_{\Sigma,1/2}  \ls \E.   
\end{equation}

\emph{Step 2:}  $G^{3,2}$ estimates
 
Our goal now is to prove the $G^{3,2}$ estimates in \eqref{nlin_ests_G_1} and \eqref{nlin_ests_G_2}.  We will estimate each of the three terms in $G^{3,2}$ separately.  

We begin by writing  
\begin{equation*}
\sigma( c_0 + c) - \sigma( c_0) = \int^{c}_{0} \sigma^{\prime} ( c_0 + s) \dd s. 
\end{equation*}
We then use the estimate \eqref{i_s_p_01} with $r= s_{1} =s_{2} = \frac{3}{2}$ to bound
\begin{eqnarray*}
& & \|  (   \sigma( c_0 + c) -  \sigma( c_0)   ) \Delta_{*} \eta   \|_{ \Sigma, \frac{3}{2}  }  \lesssim \big{\|}\int^{c}_{0} \sigma^{\prime} ( c_0 + s) \dd s  \big{\|}_{\Sigma, \frac{3}{2}}
\| \Delta_{*} \eta \|_{\Sigma, \frac{3}{2}} \\
&\lesssim& \| \sigma \|_{C^{3}} \|  c \|_{\Sigma, \frac{3}{2}} \| \eta \|_{\Sigma, \frac{7}{2}}\\
&   \lesssim& \sqrt{\mathcal{E}} \sqrt{\mathcal{D}}.
\end{eqnarray*}
Similarly, \ref{i_s_p_02} with $r= s_{1} = \frac{1}{2}$ and $ s_{2} = \frac{3}{2}+$ provides the estimate
\begin{eqnarray*}
& & \|  (   \sigma( c_0 + c) -  \sigma( c_0)   ) \Delta_{*} \eta   \|_{\Sigma, \frac{1}{2} }  \lesssim \big{\|}\int^{c}_{0} \sigma^{\prime} ( c_0 + s) \dd s  \big{\|}_{\Sigma, \frac{3}{2}+}
\| \Delta_{*} \eta \|_{\Sigma, \frac{1}{2}} \\
&\lesssim& \| \sigma \|_{C^{3}} \|  c \|_{\Sigma, \frac{3}{2}+} \| \eta \|_{\Sigma, \frac{5}{2}}\\
&   \lesssim& \sqrt{\mathcal{E}} \sqrt{\mathcal{E}} \ls \E.
\end{eqnarray*}

We now turn our attention to the second term in $G^{3,2}$ by expanding
\begin{equation*}
  H(\eta )  - \Delta_{*} \eta = \Delta_{*} \eta \left( \frac{1}{\sqrt{1+ | \nabla_{*} \eta|^{2}}} -1 \right) -
\frac{( \nabla_{*} \eta \cdot \nabla_{*}) \nabla_{*} \eta \cdot \nabla_{*} \eta}{[1+ | \nabla_{*} \eta|^{2}]^{3/2}}.
\end{equation*}
Then we use the estimate \eqref{i_s_p_01} with $r= s_{1} =s_{2} = \frac{3}{2}$ to bound
\begin{eqnarray*} 
&& \| \sigma( c_0+ c) (H(\eta) - \Delta_{*} \eta) \mathcal{N} \|_{\Sigma, \frac{3}{2}}  \\ 
&\lesssim& \|\sigma \|_{C^{2}} \| c \|_{\Sigma, \frac{3}{2}+}
\| |\nabla_{*} \eta|^{2} \|_{\Sigma, \frac{3}{2} } \| \nabla_{*}^{2} \eta \|_{\Sigma, \frac{3}{2}}
 \|  \nabla_{*} \eta \|_{\Sigma, \frac{3}{2} } \\
 &\lesssim&   \sqrt{\mathcal{E}}   \sqrt{\mathcal{D}}.
\end{eqnarray*}
Similarly we use estimate  \eqref{i_s_p_02} with $r= s_{1}= \frac{1}{2}$, $s_{2} = \frac{3}{2}+$ and \eqref{i_s_p_01} with $r=s_{1}=s_{2} = \frac{3}{2}+$ in order to show that
\begin{eqnarray*} 
&& \| \sigma( c_0 + c) (H(\eta) - \Delta_{*} \eta) \mathcal{N}] \|_{\Sigma, \frac{1}{2}}  \\
&\lesssim& \|\sigma \|_{C^{2}} \| c \|_{\Sigma, \frac{3}{2}+ }
\| |\nabla_{*} \eta|^{2} \|_{\Sigma, \frac{3}{2} +} \| \nabla_{*}^{2} \eta \|_{\Sigma, \frac{1}{2}}
 \|  \nabla_{*} \eta \|_{\Sigma, \frac{3}{2}+ } \\
 &\lesssim& \| c \|_{\Sigma, \frac{3}{2}+ }
\|  \nabla_{*} \eta  \|_{\Sigma, \frac{3}{2} +}^{2} \| \nabla_{*}^{2} \eta \|_{\Sigma, \frac{1}{2}}
 \|  \nabla_{*} \eta \|_{\Sigma, \frac{3}{2}+ }
 \\
 &\lesssim&   \sqrt{\mathcal{E}}   \sqrt{\mathcal{E}} \ls \E.
\end{eqnarray*}

Now, for the third term in $G^{3,2}$ we write  $\mathcal{N} - e_{3} = - \partial_{1} \eta e_{1} - \partial_{2} \eta e_{2}$ and then use \eqref{i_s_p_01} with $r=s_{1} =s_{2} = \frac{3}{2}$ to bound
 \begin{eqnarray*}
 &&\| \sigma( c_0 + c) \Delta_{*} \eta (\mathcal{N} - e_{3}) \|_{\Sigma, \frac{3}{2}} \lesssim  \| \sigma \|_{C^{2}} \| c \|_{\Sigma, \frac{3}{2}} \| \Delta_{*} \eta \| _{\Sigma, \frac{3}{2}}  \|  \nabla_{*} \eta \|_{\Sigma, \frac{3}{2}}  
 \lesssim  \sqrt{\mathcal{E}} \sqrt{\mathcal{D}} \sqrt{\mathcal{E}}\\
 &\lesssim& \sqrt{\mathcal{E}} \sqrt{\mathcal{D}} .
 \end{eqnarray*}
We then use \eqref{i_s_p_02} with $r=s_{1}  = \frac{1}{2}$, $s_{2} = \frac{3}{2}+$ and \eqref{i_s_p_01} with $r=s_{1} =s_{2} = \frac{3}{2}+$ to bound
 \begin{eqnarray*}
 &&\| \sigma( c_0 + c) \Delta_{*} \eta (\mathcal{N} - e_{3}) \|_{\Sigma, \frac{1}{2}} \lesssim  \| \Delta_{*} \eta \|_{\Sigma, \frac{1}{2}} \| \sigma( c_0 + c) \nabla_{*} \eta \|_{\Sigma, \frac{3}{2}+}
 \\ &\lesssim&  \|  \eta \| _{\Sigma, \frac{5}{2}}  \| \sigma \|_{C^{2}} \| c \|_{\Sigma, \frac{3}{2}+} \|  \nabla_{*} \eta \|_{\Sigma, \frac{3}{2}+}  
 \lesssim \| \sigma \|_{C^{2}} \sqrt{\mathcal{E}} \sqrt{\mathcal{E}} \sqrt{\mathcal{E}}\\
 &\lesssim& \sqrt{\mathcal{E}} \sqrt{\mathcal{E}} \ls \E.
 \end{eqnarray*}

The above analysis covers all three terms in $G^{3,2}$, and so we deduce that 
\begin{equation*}
 \norm{G^{3,2}}_{\Sigma,3/2} \ls \sqrt{\E} \sqrt{\D} \text{ and } \norm{G^{3,2}}_{\Sigma,1/2} \ls \E,   
\end{equation*}
which are the desired $G^{3,2}$ estimates in \eqref{nlin_ests_G_1} and \eqref{nlin_ests_G_2}.

\emph{Step 3:}  $G^{3,3}$ estimates
 
Now we estimate the three terms appearing in $G^{3,3}$.

To handle the first term we note that 
\begin{equation*}
[1+s^{2}]^{\frac{1}{2}}=1 + \int^{s}_{0} \frac{\tau}{[1+  \tau^{2}]^{1/2}}\dd \tau 
\end{equation*}
and 
\begin{equation*}
\sigma^{\prime}( c_0 + c) = \sigma^{\prime}( c_0 ) + \int^{c}_{0} \sigma^{\prime\prime} ( c_0 + s) \dd s. 
\end{equation*}
These combine with the estimate \eqref{i_s_p_01} with $r=s_{1}=s_{2} = \frac{3}{2}$ and yield the estimate
\begin{eqnarray*}
&& \|  (  \sqrt{1+ |\nabla_{*} \eta|^{2}}-1) \sigma^{\prime} (c_{0} + c ) \nabla_{  *} c \|_{\Sigma, \frac{3}{2}  } 
 \lesssim  \|  |\nabla_{*}\eta|^{2} \|_{\Sigma, \frac{3}{2}}  \|\sigma \|_{C^{3}} \|c \|_{\Sigma, \frac{3}{2}}  \| \nabla_{*}c \|_{\Sigma, \frac{3}{2}}\\
 & \lesssim&\|   \nabla_{*}\eta \|_{\Sigma, \frac{3}{2}}^{2}   \|\sigma \|_{C^{3}} \|c \|_{\Sigma, \frac{3}{2}}  \| \nabla_{*}c \|_{\Sigma, \frac{3}{2}} 
 \lesssim \mathcal{E} \sqrt{\mathcal{E}} \sqrt{\mathcal{D}} \\
 &\lesssim&  \sqrt{\mathcal{E}} \sqrt{\mathcal{D}}.
\end{eqnarray*}
We similarly use \eqref{i_s_p_02} with $r= s_{1}= \frac{1}{2}$, $s_{2} = \frac{3}{2}+$ and \eqref{i_s_p_01} with $r= s_{1}= s_{2} = \frac{3}{2}+$ to see that
\begin{eqnarray*}
&& \|  (  \sqrt{1+ |\nabla_{*} \eta|^{2}}-1) \sigma^{\prime} ( c_0 + c ) \nabla_{  *} c \|_{\Sigma, \frac{1}{2}  } \\
&\lesssim&
\|  (  \sqrt{1+ |\nabla_{*} \eta|^{2}}-1) \sigma^{\prime} ( c_0 + c ) \|_{\Sigma, \frac{3}{2}+}\|\nabla_{  *} c  \|_{\Sigma, \frac{1}{2}}\\
&\lesssim&
\|  (  \sqrt{1+ |\nabla_{*} \eta|^{2}}-1)\|_{\Sigma, \frac{3}{2}+} \| \sigma^{\prime} ( c_0 + c ) \|_{\Sigma, \frac{3}{2}+}\|\nabla_{  *} c  \|_{\Sigma, \frac{1}{2}} 
\\
&
 \lesssim  &\|  |\nabla_{*}\eta|^{2} \|_{\Sigma, \frac{3}{2}+}  \|\sigma \|_{C^{3}} \|c \|_{\Sigma, \frac{3}{2}+}  \| \nabla_{*}c \|_{\Sigma, \frac{1}{2}}\\
 & \lesssim&\|   \eta \|_{\Sigma, \frac{5}{2}+}^{2}  \| c \|_{\Sigma, 2} ^{2}
 \lesssim  \| \eta \|_{\Sigma,3}^{2 }  \| c \|_{\Sigma, 2} ^{2} \lesssim \mathcal{E}\mathcal{E} \\
 &\lesssim& \E.
\end{eqnarray*}

To estimate the second term in $G^{3,3}$ we use \eqref{i_s_p_01} with $r=s_{1} = s_{2} = \frac{3}{2}$: 
\begin{eqnarray*}
\|   ( \sigma^{\prime} ( c_0 + c) - \sigma^{\prime} ( c_0)) \nabla_{*} c\|_{ \Sigma,\frac{3}{2}} &\lesssim& \|  \sigma^{\prime}( c_0 + c) - \sigma^{\prime} ( c_0) \|_{\Sigma, \frac{3}{2}} \|c \|_{\Sigma, \frac{5}{2}}
\lesssim \|\sigma \|_{C^{4}} \| c \|_{\Sigma, \frac{3}{2}} \| c\|_{\Sigma, \frac{5}{2}}\\
& \lesssim& \sqrt{\mathcal{E}}\sqrt{\mathcal{D}}.
\end{eqnarray*}
Also, we use \eqref{i_s_p_02} with $r=s_{1} = \frac{1}{2}$ and $s_{2} = \frac{3}{2}+$ to bound
\begin{eqnarray*}
\|   ( \sigma^{\prime} ( c_0 + c) - \sigma^{\prime} ( c_0)) \nabla_{*} c\|_{ \Sigma,\frac{1}{2}} &\lesssim& \|  \sigma^{\prime}( c_0+ c) - \sigma^{\prime} ( c_0) \|_{\Sigma, \frac{3}{2}+} \|c \|_{\Sigma, \frac{3}{2}}
\lesssim \|\sigma \|_{C^{2}} \| c \|_{\Sigma, \frac{3}{2}+} \| c\|_{\Sigma, \frac{3}{2}}\\
& \lesssim& \sqrt{\mathcal{E}}\sqrt{\mathcal{E}} \ls \E.
\end{eqnarray*}

For the third term we first write  $\nabla_{\Sigma,*} c - \nabla_{*} c = - \nu_{*} (\nu_{*} \cdot \nabla_{*}) c$.  Then  \eqref{i_s_p_01} with $r=s_{1} = s_{2} = \frac{3}{2}$ tells us that
\begin{eqnarray*}
&&\| \sqrt{1+ |\nabla_{*} \eta|^{2}} \sigma^{\prime}(c_0 + c)( \nabla_\Gamma c -  \nabla_{*} c  ) \|_{\Sigma, \frac{3}{2}} 
 \lesssim  \| \sigma^{\prime}(c_0 + c) \|_{\Sigma, \frac{3}{2}} \|  \nabla_{*} \eta (   \nu_{*} \cdot \nabla_{*} c  ) \|_{\Sigma, \frac{3}{2}} 
\\
 &
\lesssim&
\| \sigma \|_{C^{3}} \| c \|_{\Sigma, \frac{3}{2}}
 \| \eta \|_{\Sigma, \frac{5}{2}}^{2} \| c\|_{\Sigma, \frac{5}{2}} 
\lesssim \sqrt{\mathcal{E}} \sqrt{\mathcal{D}}.
\end{eqnarray*}
Similarly,  \eqref{i_s_p_02} with $r=s_{1} = \frac{1}{2}$, $s_{2} = \frac{3}{2}+$ and \eqref{i_s_p_01} with $r=s_{1} = s_{2} = \frac{3}{2}+$ imply that
\begin{eqnarray*}
&&\| \sqrt{1+ |\nabla_{*} \eta|^{2}} \sigma^{\prime}(c + c_0)( \nabla_\Gamma c -  \nabla_{*} c  ) \|_{\Sigma, \frac{1}{2}} 
 \lesssim  \| \sigma^{\prime}(c + c_0)  \nabla_{*} \eta \nu_{*} \|_{\Sigma, \frac{3}{2}+} \|    \nabla_{*} c   \|_{\Sigma, \frac{1}{2}} 
\\
 &
\lesssim&
\| \sigma \|_{C^{3}} \| c \|_{\Sigma, \frac{3}{2}+}
 \| \eta \|_{\Sigma, \frac{5}{2}+}^{2} \| c\|_{\Sigma, \frac{3}{2}} 
\lesssim \sqrt{\mathcal{E}} \sqrt{\mathcal{E}}.
\end{eqnarray*}

The above analysis covers all three terms in $G^{3,3}$, and so we deduce that 
\begin{equation*}
 \norm{G^{3,3}}_{\Sigma,3/2} \ls \sqrt{\E} \sqrt{\D} \text{ and } \norm{G^{3,3}}_{\Sigma,1/2} \ls \E,   
\end{equation*}
which are the desired $G^{3,3}$ estimates in \eqref{nlin_ests_G_1} and \eqref{nlin_ests_G_2}.

\emph{Step 4:}  $G^{3,4}$ estimates
 
For $G^{3,4}$ we have the estimates
\begin{equation*}
 \norm{G^{3,4}}_{\Sigma,3/2} \ls \sqrt{\E} \sqrt{\D} \text{ and } \norm{G^{3,4}}_{\Sigma,1/2} \ls \E,   
\end{equation*}
which are the desired estimates in \eqref{nlin_ests_G_1} and \eqref{nlin_ests_G_2}.  These bounds follow from the same, if not somewhat simpler, arguments used to bound $G^{3,3}$ and are thus omitted for the sake of brevity.

\emph{Step 5:}  $G^{5}$ estimates 

There are four terms in $G^5$.  We handle the first pair with the Sobolev embedding and trace theory:
\begin{eqnarray*}
&&\| u \cdot \nabla_{*} c + c \diverge_\Gamma u \|_{\Sigma,1}\\
&\lesssim& (1+ \| \nabla_{*} \eta \|_{\infty}) \| \nabla  u  \|_{L^{4}(\Sigma) } \| \nabla_{*}c \|_{L^{4}(\Sigma)  }
+ \| u \|_{L^{4}(\Sigma) } \| \nabla_{*}^{2} c \|_{L^{4}(\Sigma) } \\
&&
+ \| c \|_{L^{\infty}(\Sigma)} \| \nabla_{*}^{2} \eta \|_{ L^{4}(\Sigma)} \| u \|_{ L^{4}(\Sigma)}
 + \| c \|_{L^{\infty}(\Sigma)} \| \nabla_{*} \eta \|_{ L^{4}(\Sigma)} \| \nabla u \|_{ L^{4}(\Sigma)}\\
 &\lesssim& \sqrt{\mathcal{E}}\sqrt{\mathcal{D}}.
\end{eqnarray*}

For the third term we use the Sobolev embeddings $H^{1}(\Sigma) \hookrightarrow L^{4}(\Sigma) $ and $H^{1+}(\Sigma) \hookrightarrow L^{\infty}(\Sigma)$ to bound
\begin{eqnarray*}
&&\|  \gamma \Delta_\Gamma c - \gamma \Delta_{*} c\|_{\Sigma, 1} \lesssim \|  \nabla_{*}  \eta  \nabla_{*}^{2} c \|_{\Sigma, 1}  +  \|  \nabla_{*}^{2}  \eta  \nabla_{*} c \|_{\Sigma, 1}\\
& \lesssim &\| \eta \|_{\Sigma, \frac{5}{2}} \| c \|_{\Sigma,3}+ \| \eta \|_{\Sigma, 3} \| c \|_{\Sigma, \frac{5}{2}} + \| \nabla_{*}^{2} \eta \|_{L^{4}(\Sigma)} \| \nabla_{*} ^{2}c \|_{L^{4}(\Sigma)}  \\
&\lesssim& \sqrt{\mathcal{E}}  \sqrt{\mathcal{D}}.
\end{eqnarray*}

For the fourth term in $G^5$ we first note that
\begin{eqnarray*}
 \diverge_\Gamma u - \text{div}_{*} u 
&=&\{ \text{div}_{*} u_{*} - \nu_{*} \cdot (\nu_{*} \cdot \nabla_{*}) u_{*} - \nu_{3} (\nu_{*} \cdot \nabla_{*}) u_{3}\} - \text{div}_{*} u_{*}  \\ 
&=&- \nu_{*} \cdot (\nu_{*} \cdot \nabla_{*}) u_{*}   - \nu_{3} (\nu_{*} \cdot \nabla_{*}) u_{3} .
\end{eqnarray*}
Then by trace theory and the Sobolev embedding imply that
\begin{eqnarray*}
&&\| c_{0} [\diverge_\Gamma u - \text{div}_{*} u ] \|_{\Sigma, 1}  \lesssim  \| \nabla_{*}^{2} \eta \nabla u \|_{\Sigma, 0} + \| \nabla_{*} \eta \nabla^{2} u \|_{\Sigma,0} \lesssim  \|   \eta \|_{\Sigma, 3} \| \nabla u \|_{\Sigma, 1}\\
& \lesssim&  \|   \eta \|_{\Sigma, 3} \|  u \|_{\frac{5}{2}+} \lesssim \sqrt{\mathcal{E}}  \sqrt{\mathcal{D}}. 
\end{eqnarray*}

Combining the above analysis, we deduce that 
\begin{equation*}
  \norm{G^5}_{\Sigma,1} \ls \sqrt{\E} \sqrt{\D},  
\end{equation*}
which is the desired $G^5$ estimate in \eqref{nlin_ests_G_1}.

\end{proof}

\begin{remark}
 It is in the $G^3$ estimates of this result that we need the full power of the assumption $\sigma \in C^3$ from \eqref{sigma_assume}.
\end{remark}

\subsection{The average of $c$  }

We now aim to estimate $\int_\Sigma c$ as a nonlinear term.

\begin{prop}\label{nlin_ests_c}
Let $\E$ and $\D$ be as defined in \eqref{def_E} and \eqref{def_D}. Suppose that $\E \le \delta$, where $\delta \in (0,1)$ is the universal constant given in Lemma \ref{infinity_bounds}, and that $\D < \infty$.  Then 
\begin{equation}\label{nlin_ests_c_0}
 \abs{\int_\Sigma c} \ls \sqrt{\E} \sqrt{\D}.
\end{equation}
\end{prop}
\begin{proof}
We use \eqref{conserv_c} to compute
\begin{multline*}
 \int_\Sigma c = \int_\Sigma (-c_0 + \tilde{c}) = - c_0 \abs{\Sigma} + \int_{\Sigma } \tilde{c} \sqrt{1+ \abs{\nab_\ast \eta }^2} + \tilde{c} \left(1-\sqrt{1+ \abs{\nab_\ast \eta }^2} \right) \\ = -c_0\abs{\Sigma} + c_0 \abs{\Sigma} 
 + \int_\Sigma (c_0 + c) \left(1-\sqrt{1+ \abs{\nab_\ast \eta }^2} \right) = \int_\Sigma (c_0 + c) \left(1-\sqrt{1+ \abs{\nab_\ast \eta }^2} \right).
\end{multline*}
From this and the trivial bound  $0 \le \sqrt{1+x^2} -1 \le x^2$ we deduce that
\begin{equation*}
 \abs{\int_\Sigma c} \le \int_\Sigma (c_0 + \abs{c}) \abs{\nab_\ast \eta}^2 \ls (c_0 + \norm{c}_{L^\infty}) \ns{\eta}_{1} \ls (1+\sqrt{\E})\sqrt{\E} \sqrt{D} \ls \sqrt{\E} \sqrt{D}.
\end{equation*}
This proves \eqref{nlin_ests_c_0}.

\end{proof}

\section{A priori estimates  }\label{sec_aprioris}

In this section we combine energy-dissipation estimates with various elliptic estimates and estimates of the nonlinearities in order to deduce a system of a priori estimates.

\subsection{Energy-dissipation estimates  }
 
In order to state our energy-dissipation estimates we must first introduce some notation.  Recall that for a multi-index $\alpha = (\alpha_0,\alpha_1,\alpha_2,\alpha_3) \in \mathbb{N}^{1+3}$ we write $\abs{\alpha} = 2 \alpha_0 + \alpha_1 + \alpha_2 + \alpha_3$ and $\p^\alpha = \p_t^{\alpha_0} \p_1^{\alpha_1}  \p_2^{\alpha_2}\p_3^{\alpha_3}$.  For $\alpha \in \mathbb{N}^{1+3}$ we set 
 \begin{equation}\label{bar_ED_alpha}
 \begin{split}
 \bar{\mathcal{E}}_\alpha &: =  \int_{\Omega} \hal  \abs{\partial^{\alpha} u }^2 
+ \int_{\Sigma} \hal \abs{\partial^{\alpha} \eta}^2  + \frac{\sigma_{0}}{2} \abs{\nabla_{*} \partial^{\alpha} \eta}^{2 } + \frac{-  \sigma^{\prime}_{0} }{2c_{0}} \abs{ \partial^{\alpha} c}^{2}, \\
 \bar{\mathcal{D}}_\alpha &: =  \int_{\Omega}\hal  \abs{ \mathbb{D} \partial^{\alpha} u}^{2} 
 +  \int_{\Sigma} \frac{- \gamma \sigma_{0}^{\prime}}{c_{0}} \abs{\nabla_{*} \partial^{\alpha} c}^{2}.
 \end{split}
\end{equation}  
We then define
\begin{equation}\label{bar_ED}
 \bar{\mathcal{E}} := \sum_{|\alpha| \leq 2} \bar{\E}_\alpha \text{ and } 
 \bar{\mathcal{D}} := \sum_{|\alpha| \leq 2} \bar{\D}_\alpha.
\end{equation}  
We will also need to use the functional
\begin{equation}\label{F_def}
 \mathcal{F} := \int_\Omega p F^2 J,
\end{equation}
where $F^2$ is as defined in \eqref{nlin_F2}.

Our next result encodes the energy-dissipation inequality associated to $\bar{\E}$ and $\bar{\D}$.

\begin{thm}\label{energy_ev}
Suppose that $(u,p,\eta,c)$ solves \eqref{ns_geometric} on the temporal interval $[0,T]$.  Let $\E$ and $\D$ be as defined in \eqref{def_E} and \eqref{def_D}, and suppose that 
\begin{equation*}
 \sup_{0\le t \le T} \E(t) \le \delta \text{ and } \int_0^T \D(t) dt < \infty,
\end{equation*}
where $\delta \in (0,1)$ is the universal constant given in Lemma \ref{infinity_bounds}.  Let $\seb$ and $\sdb$ be given by \eqref{bar_ED} and $\mathcal{F}$ be given by \eqref{F_def}.  Then 
\begin{equation}\label{energy_ev_0}
 \frac{d}{dt} \left(\bar{\E} -\mathcal{F} \right) +  \bar{\mathcal{D}}  \lesssim     \sqrt{\mathcal{E}}  \mathcal{D}
\end{equation}
for all $t \in [0,T]$.
\end{thm}
\begin{proof}

Let $\alpha \in \mathbb{N}^{1+3}$ with $\abs{\alpha}\le 2$.  We apply $\p^\alpha$ to \eqref{ns_geometric} to derive an equation for $(\p^\alpha u,\p^\alpha p,\p^\alpha \eta,\p^\alpha c)$.  We will consider the form of this equation in different ways depending on $\alpha$.

Suppose that $\abs{\alpha} =2$ and $\alpha_0 =1$, i.e. that $\p^\alpha = \p_t$.  Then $v = \dt u$, $q = \dt p$, $\zeta = \dt \eta$, and $h = \dt c$ satisfy \eqref{ns_geometric_hybrid} with $F^1,\dotsc,F^5$ as given in \eqref{nlin_F1}--\eqref{nlin_F5}.  According to Proposition \ref{ed_geometric} we then have that 
\begin{multline*} 
 \frac{d}{dt} \left( \int_{\Omega} \frac{\abs{\dt u}^{2}}{2} J +  \int_{\Sigma} \frac{ \abs{\dt \eta}^{2}}{2} 
+ \int_{\Sigma} \sigma_{0} \frac{\abs{\nabla_{*} \dt \eta}^{2}}{2} + 
\frac{- \sigma_{0}^{\prime}}{c_0} \int_{\Sigma} \frac{\abs{\dt c}^{2}}{2}
 \right) + \int_{\Omega} \frac{\abs{\mathbb{D}_{\mathcal{A}} \dt u}^{2}}{2} J
 +  \frac{- \gamma \sigma_{0}^{\prime}}{c_0} \int_{\Sigma} \abs{\nabla_{*} \dt c}^{2} 
\\
=  \int_{\Omega}  ( \dt u\cdot F^{1} + \dt p F^{2})J + \int_{\Sigma} -\dt u\cdot F^{3}   + \int_{\Sigma} (\dt \eta - \sigma_{0} \Delta_{*} \dt \eta) F^{4} +  \frac{- \sigma_{0}^{\prime}}{c_0} \int_{\Sigma} \dt c \cdot F^{5}.
\end{multline*}
We then write 
\begin{equation*}
 \int_{\Omega} \dt p F^2 J = \frac{d}{dt} \int_\Omega p F^2 J - \int_\Omega p (\dt F^2 J + F^2 J),
\end{equation*}
collect the temporal derivative terms,  and then apply the estimates \eqref{nlin_ests_F_01}--\eqref{nlin_ests_F_03} of Theorem \ref{nlin_ests_F}, the estimates of Lemma \ref{infinity_bounds}, and the usual trace estimates to deduce that 
\begin{equation}\label{energy_ev_2}
\frac{d}{dt}\left( \seb_{(1,0,0,0)} - \F \right) + \sdb_{(1,0,0,0)}  \ls   \sqrt{\mathcal{E}}  \mathcal{D},
\end{equation}
where $\seb_{(1,0,0,0)}$ and $\sdb_{(1,0,0,0)}$ are as defined in \eqref{bar_ED_alpha}.

Next we consider $\alpha \in \mathbb{N}^{1+3}$ with $\alpha_0 =0$, i.e. no temporal derivatives.  In this case we view $(u,p,\eta,c)$ in terms of \eqref{ns_perturbed}, which then means that  $(v,q,\zeta,h) = (\p^\alpha u, \p^\alpha p, \p^\alpha \eta, \p^\alpha c)$ satisfy  \eqref{eqtn_v} with $\Phi^i = \p^\alpha G^i$ for $i=1,\dotsc,5$, where the nonlinearities $G^i$ are as defined in \eqref{nlin_G1}--\eqref{nlin_G5}.  We may then apply Proposition \ref{ed_pert} to see that for $\abs{\alpha} \le 2$ and $\alpha_0 =0$ we have the identity
\begin{multline}\label{energy_ev_1}
\frac{d}{dt} \bar{\E}_\alpha + \bar{\D}_\alpha 
= \int_{\Omega} \p^\alpha u \cdot \p^\alpha G^{1} + \p^\alpha p \p^\alpha G^{2} - \p^\alpha u \cdot \nabla \p^\alpha G^{2} \\
+ \int_{\Sigma} - \p^\alpha u \cdot \p^\alpha G^{3} + \p^\alpha \eta \p^\alpha G^{4} - \sigma_{0} \p^\alpha G^{4} \Delta_{*} \p^\alpha \eta- \frac{\sigma^{\prime}_{0}}{ c_0} \p^\alpha G^{5} \p^\alpha c .
\end{multline}

When $\abs{\alpha}=2$ and $\alpha_0=0$ we write $\partial^\alpha = \partial^\beta \partial^\omega$ for $\abs{\beta} = \abs{\omega} =1$.  We then integrate by parts in the $G^1,G^4,$ and $G^5$ terms in \eqref{energy_ev_1} to estimate 
\begin{multline*}
 \text{RHS of } \eqref{energy_ev_1}
=  \int_{\Omega} -\p^{\alpha+\beta} u \cdot \p^\omega G^{1} + \p^\alpha p \p^\alpha G^{2} - \p^\alpha u \cdot \nabla \p^\alpha G^{2} \\
+ \int_{\Sigma} - \p^\alpha u \cdot \p^\alpha G^{3} - \p^{\omega} \eta \p^{\alpha+\beta} G^{4} + \sigma_{0} \p^{\alpha+\beta} G^{4} \Delta_{*} \p^{\omega} \eta+ \frac{\sigma^{\prime}_{0}}{ c_0} \p^{\alpha+\beta} G^{5} \p^{\omega} c \\
\ls \|  u \|_{3} \| G^{1}\|_{1} + \| p \|_{2} \| G^{2} \|_{2} + \| u \|_{3} \| G^{2} \|_{2}\\
+ \| D^{2} u \|_{\Sigma, \frac{1}{2}} \| D^{2} G^{3} \|_{\Sigma, - \frac{1}{2}}
+ \| D^{3} G^{4} \|_{\Sigma,  -\frac{1}{2}}
[ \| D \eta  \|_{\Sigma, \frac{1}{2}} +\| D^{3} \eta \|_{\Sigma, \frac{1}{2}}]
 + \| D^{2} G^{5} \|_{\Sigma, 0} \| D^{2} c\|_{\Sigma, 0}\\
 \lesssim  \sqrt{\mathcal{D}} \left\{
  \| G^{1}\|_{1} + \| G^{2} \|_{2} + \| G^{3} \|_{\Sigma, \frac{3}{2}}
  + \| G^{4} \|_{\Sigma, \frac{5}{2}} + \| G^{5} \|_{\Sigma,1}
 \right\}.
\end{multline*}
The estimate \eqref{nlin_ests_G_01} of Theorem  \ref{nlin_ests_G} then tells us that 
\begin{equation*}
  \text{RHS of } \eqref{energy_ev_1} \ls \sqrt{\E} \D,
\end{equation*}
and so we find  that for  $\seb_{\alpha}$ and $\sdb_{\alpha}$ as in \eqref{bar_ED_alpha} we have the inequality
\begin{equation}\label{energy_ev_3}
 \frac{d}{dt} \sum_{\substack{\abs{\alpha}=2 \\\alpha_0 =0}}  \seb_{\alpha}  +  \sum_{\substack{\abs{\alpha}=2 \\\alpha_0 =0}} \sdb_{\alpha}  \ls \sqrt{\E} \D.
\end{equation}

On the other hand, if $\abs{\alpha} < 2$ then we must have that $\alpha_0 =0$, and we can directly apply Theorem  \ref{nlin_ests_G} to see that 
\begin{equation*}
  \text{RHS of } \eqref{energy_ev_1} \ls \sqrt{\E} \D.
\end{equation*}
From this we deduce that 
\begin{equation}\label{energy_ev_4}
 \frac{d}{dt} \sum_{\abs{\alpha} \le 1}  \seb_{\alpha}  +   \sum_{\abs{\alpha} \le 1}  \sdb_{\alpha}  \ls  \sqrt{\E} \D.
\end{equation}

Now, to deduce \eqref{energy_ev_0} we simply sum \eqref{energy_ev_2}, \eqref{energy_ev_3}, and \eqref{energy_ev_4}.

\end{proof}

\subsection{Enhanced energy estimates  }

From the energy-dissipation estimate of Theorem \ref{energy_ev} we have control of $\bar{\E}$ and $\bar{\D}$.  Our goal now is to show that these can be used to control $\E$ and $\D$ up to some error terms that we will be able to guarantee are small.  Here we focus on the estimate for the energies, $\bar{\E}$ and $\E$.

\begin{thm}\label{enhance_energy}
Let $\E$ be as defined in \eqref{def_E}. Suppose that $\E \le \delta$, where $\delta \in (0,1)$ is the universal constant given in Lemma \ref{infinity_bounds}.  Then 
\begin{equation}\label{enhance_energy_0}
\mathcal{E} \lesssim \bar{\mathcal{E}} + \mathcal{E}^{2}.
\end{equation}
\end{thm}
\begin{proof}

According to the definitions of $\bar{\E}$ and $E$, in order to prove \eqref{enhance_energy_0} it suffices to prove that 
\begin{equation}\label{enhance_energy_1}
\| u \|_{2}^{2} +\| p \|_{1}^2 + \| \partial_{t} \eta  \|_{\Sigma, \frac{3}{2}}^2 + \| \partial_{t}^{2} \eta \|_{\Sigma, -\frac{1}{2}}^2 \ls  \bar{\mathcal{E}} + \mathcal{E}^{2}. 
\end{equation}

For estimating $u$ and $p$ we recall the standard Stokes estimate (see for instance \cite{GT2}): for $r \geq 0$,
\begin{equation}\label{est_stokes}
\| u \|_{r}  + \| p \|_{r-1} \lesssim \| \phi \|_{r-2} + \| \psi\|_{r-1} + \| \alpha \|_{\Sigma, r-\frac{3}{2}},
\end{equation}
if
\begin{equation*}
\begin{cases}
-\Delta u + \nabla p = \phi \in H^{r-2}(\Omega)  \\
\mathrm{div}  \ {v}= \psi \in H^{r-1}(\Omega)&  \\
(pI - \mathbb{D}u ) e_{3}  = \alpha \in H^{r-\frac{3}{2}}(\Sigma)&   \\
u|_{\Sigma_{b}} = 0. 
\end{cases}
\end{equation*}
Now, according to  \eqref{ns_perturbed} we have that
\begin{equation} \notag
\begin{cases}
   -\Delta u   + \nabla  p = - \partial_{t} u+  G^{1}& \text{in }
\Omega  \\
\text{div} \ {u}=G^{2} & \text{in }\Omega  \\
(pI - \mathbb{D}u   ) e_{3}  
= (\eta I + \sigma_{0} \Delta_{*} \eta) e_{3}- \sigma_{0}^{\prime} \nabla_{*} c + 
  G^{3} & \text{on } \Sigma \\
u = 0 & \text{on } \Sigma_{b},
\end{cases}
\end{equation}
and hence we may apply \eqref{est_stokes} and the estimate \eqref{nlin_ests_G_02} of  Theorem \ref{nlin_ests_G} to see that 
\begin{eqnarray*}
\| u \|_{2} + \|  p\|_{1} &\lesssim& \| \partial_{t} u \|_{0} + \| G^{1} \|_{0} + \| G^{2} \|_{1} + \|  (\eta I + \sigma_{0} \Delta_{*} \eta) e_{3}- \sigma_{0}^{\prime} \nabla_{*} c\|_{\Sigma, \frac{1}{2}}
+ \| G^{3} \|_{\Sigma, \frac{1}{2}}\\
&\lesssim& \sqrt{\bar{\mathcal{E}}}    + \| G^{1} \|_{0} + \| G^{2} \|_{1} + \| G^{3} \|_{\Sigma, \frac{1}{2}}\\
&\lesssim&   \sqrt{\bar{\mathcal{E}}}  + \mathcal{E}.
\end{eqnarray*}
From this we deduce that the $u,p$ estimates in \eqref{enhance_energy_1} hold.
 
To estimate the $\dt \eta$ term in \eqref{enhance_energy_1} we use the fourth equation of \eqref{ns_perturbed} in conjunction with the  estimate \eqref{nlin_ests_G_02} of  Theorem \ref{nlin_ests_G} and the usual trace estimates to see that 
\begin{equation*}
 \| \partial_{t} \eta \|_{\Sigma, \frac{3}{2}}   \lesssim   \| u_{3} \|_{\Sigma, \frac{3}{2}} + \| G^{4} \|_{\Sigma, \frac{3}{2}}   \lesssim   
 \|  u \|_{2}   + \E
   \lesssim  
 \sqrt{\bar{\mathcal{E}}} +   \mathcal{E}. 
\end{equation*}
From this we deduce that the $\dt \eta$ estimate in \eqref{enhance_energy_1} holds.

It remains only to estimate the $\dt^2 \eta$ term in \eqref{enhance_energy_1}.  For this we apply a temporal derivative to the fourth equation of \eqref{ns_perturbed} and integrate against a  function $\phi \in H^{1/2}(\Sigma)$ to see that 
\begin{equation*}
  \int_{\Sigma} \partial_{t}^{2} \eta \phi \dd x_{*} =   \int_{\Sigma} \partial_{t} u_{3} \phi \dd x_{*}  +  \int_{\Sigma} \partial_{t}  G^{4} \phi \dd x_{*} . 
\end{equation*}
Choose an extension $E\phi \in H^{1}(\Omega)$ with $E\phi|_{\Sigma} = \phi$, $E \phi|_{\Sigma_{b}}=0,$ and $\|E\phi\|_{1} \ls \| \phi \|_{\Sigma, \frac{1}{2}}$.  Then 
\begin{equation*}
\int_{\Sigma} \partial_{t} u _{3} \phi =   \int_{\Omega}   \partial_{t}u  \cdot \nabla_{x} E \phi
 +  \int_{\Omega} \partial_{t}G^{2}  E \phi  
\le    \left( \| \partial_{t} u \|_{0}  + \|   \partial_{t}G^{2} \|_{0} \right) \| \phi \|_{\Sigma, \frac{1}{2}}, 
\end{equation*}
and so again Theorem \ref{nlin_ests_G} implies that 
\begin{equation*}
\| \partial_{t}^{2} \eta \|_{\Sigma, -\frac{1}{2}} \  \lesssim   \ 
\| \partial_{t} u \|_{0} + \| \partial_{t} G^{2} \|_{0}
+ \| \partial_{t} G^{4} \|_{\Sigma, -\frac{1}{2}}  \ 
\lesssim    \  \sqrt{ \bar{\mathcal{E}} }+ {\mathcal{E}}.
\end{equation*}
From this we deduce that the $\dt^2 \eta$ estimate in \eqref{enhance_energy_1} holds.

\end{proof}

\subsection{Enhanced dissipation estimates  }

We now complement Theorem \ref{enhance_energy} by proving a corresponding result for the dissipation. 
 
\begin{thm}\label{enhance_dissipation}
Let $\E$, and $\D$ be as defined in \eqref{def_E} and \eqref{def_D}. Suppose that $\E \le \delta$, where $\delta \in (0,1)$ is the universal constant given in Lemma \ref{infinity_bounds}, and suppose that $\D < \infty$.  Then 
\begin{equation}\label{enhance_dissipation_0}
\mathcal{D} \lesssim \bar{\mathcal{D}} + \mathcal{E} \D.
\end{equation}
\end{thm}
\begin{proof}

Recall the Stokes elliptic estimate for the Stokes problem with Dirichlet boundary conditions (see for instance \cite{temam}): for $r \geq 2$,
\begin{equation}\label{stokes_WTK}
\| u \|_{r} + \| \nabla p \|_{r-2}  \lesssim \| f \|_{r-2} + \| h \|_{r-1} + \| \varphi_{1} \|_{\Sigma, r- \frac{1}{2}}+ \| \varphi_{2} \|_{\Sigma_b, r- \frac{1}{2}},
\end{equation}
if
\begin{equation*}
\begin{cases}
   -\Delta u   + \nabla  p =  f& \text{in }
\Omega  \\
\text{div} \ {u}=h & \text{in }\Omega  \\
u = \varphi_{1} & \text{on } \Sigma \\
u = \varphi_{2} & \text{on } \Sigma_{b}.
\end{cases}
\end{equation*}
We know that 
\begin{equation*}
\|  u \|_{1} +  \| \nabla_{*} u \|_{1} + \| \nabla_{*}^{2} u \|_{1} \lesssim \sqrt{\bar{\mathcal{D}}}, 
\end{equation*}
and so trace theory provides us with the estimate 
\begin{equation*}
 \| u \|_{\Sigma, \frac{5}{2}}   \lesssim \sqrt{\bar{\mathcal{D}}}.
\end{equation*}
We also have that  $\| \p_{t} u \|_{1} \lesssim \sqrt{\bar{\mathcal{D}}}$, and Theorem \ref{nlin_ests_G} tells us that 
\begin{equation*}
 \| G^{1} \|_{1} + \| G^{2} \|_{2} \lesssim \sqrt{\mathcal{E}}\sqrt{\mathcal{D}}.
\end{equation*}
We may thus apply \eqref{stokes_WTK} with $r=3$ and $f= - \p_{t} u + G^{1}$, $h = G^{2},$ $\varphi_{1} = u|_{\Sigma}$, and $\varphi_{2} =0$ to obtain
\begin{equation}\label{enhance_dissipation_1}
\| u \|_{3} + \| \nabla p \|_{1} \lesssim \|  - \p_{t} u + G^{1} \|_{1}
+ \| G^{2}\|_{2} + \| u \|_{\Sigma, \frac{5}{2}}
\lesssim \sqrt{\bar{\mathcal{D}}} + \sqrt{\mathcal{E}}\sqrt{\mathcal{D}}. 
\end{equation}

We now turn to the $\eta$ estimates.  For $\alpha \in \mathbb{N}^2$ with $\abs{\alpha}=1$ we apply $\p^\alpha$ to the third equation of \eqref{ns_perturbed} to obtain
\begin{equation*}
(1- \sigma_{0} \Delta_{*}) \p^\alpha \eta= \p^\alpha p - \p_{3} \p^\alpha u_{3} - \p^\alpha G_{3}^{3}. 
\end{equation*}
Then standard elliptic estimates and the trace estimates imply that
\begin{eqnarray*}
\| \nabla_{*} \eta \|_{\Sigma, \frac{5}{2}} &=& \sum_{\abs{\alpha}=1} \| \p^\alpha \eta \|_{\Sigma, \frac{5}{2}}   \\ 
  &\lesssim& \sum_{\abs{\alpha}=1} \| \p^\alpha p - \p_{3} \p^\alpha u_{3} - \p^\alpha G_{3}^{3}   \|_{\Sigma,\frac{1}{2}} \\
 &\lesssim& \| \nabla p \|_{1} + \| u \|_{3} + \| G^{3} \|_{\Sigma, \frac{3}{2}}\\
 &\lesssim&  \sqrt{\bar{\mathcal{D}}} + \sqrt{\mathcal{E}}\sqrt{\mathcal{D}}.
\end{eqnarray*}
We know from \eqref{z_avg} that $\eta$ has zero average, so   the Poincar\'e inequality tells us that $ \| \eta \|_{\Sigma,0}   \lesssim   \| \nabla_{*} \eta \|_{\Sigma, 0}$, and hence 
\begin{equation}\label{enhance_dissipation_2}
\| \eta \|_{\Sigma, \frac{7}{2}} \ls\| \eta \|_{\Sigma,0} + \| \nab_{*} \eta \|_{\Sigma, \frac{5}{2}} \ls \| \nab_{*} \eta \|_{\Sigma, \frac{5}{2}}     \lesssim    \sqrt{\bar{\mathcal{D}}} + \sqrt{\mathcal{E}}\sqrt{\mathcal{D}}.
\end{equation}

To estimate the temporal derivatives of $\eta$ we use the fourth equation in \eqref{ns_perturbed}, the estimates of Theorem \ref{nlin_ests_G}, and \eqref{enhance_dissipation_1}:
\begin{equation}\label{enhance_dissipation_3}
\| \partial_{t} \eta \|_{\Sigma, \frac{5}{2}}  \leq  \|   u_{3} \|_{\Sigma, \frac{5}{2}} + \|   G^{4}\|_{\Sigma, \frac{5}{2}}  \lesssim  \|  u  \|_{3}  + \|  G^{4}\|_{\Sigma, \frac{5}{2}} \lesssim  \sqrt{\bar{\mathcal{D}}} + \sqrt{\mathcal{E}}\sqrt{\mathcal{D}}
\end{equation}
and
\begin{equation}\label{enhance_dissipation_4}
 \| \partial_{t}^{2} \eta \|_{\Sigma, \frac{1}{2}}  \leq  \| \partial_{t} u_{3} \|_{\Sigma, \frac{1}{2}} + \| \partial_{t} G^{4}\|_{\Sigma, \frac{1}{2}}  \lesssim  \| \partial_{t} u  \|_{1}  + \| \partial_{t} G^{4}\|_{\Sigma, \frac{1}{2}} 
\lesssim  \sqrt{\bar{\mathcal{D}}} + \sqrt{\mathcal{E}}\sqrt{\mathcal{D}} .
 \end{equation}

Now we complete the estimate of the pressure by obtaining a bound for  $\|p \|_{0}$.  To this end we combine the estimates \eqref{enhance_dissipation_1} and \eqref{enhance_dissipation_2} with  the Stokes estimate of \eqref{est_stokes} with $\phi =- \p_{t}u + G^{1}, \psi = G^{2},$ and $\alpha = (\eta  I - \sigma_{0} \Delta_{*} \eta ) e_{3} - \sigma_{0}^{\prime} \nabla_{*}c+ G^{3} e_{3}$ to bound
 \begin{eqnarray*}
 \| u \|_{3} + \| p \|_{2 } &\lesssim& \| - \p_{t}u + G^{1} \|_{1} + \| G^{2}\|_{2} 
 + \| (\eta  I - \sigma_{0} \Delta_{*} \eta ) e_{3} - \sigma_{0}^{\prime} \nabla_{*}c+ G^{3} e_{3} \|_{\Sigma, \frac{3}{2}}\\
 &\lesssim& \| \p_{t} u \|_{1} + \| G^{1} \|_{1} + \|  G^{2} \|_{2} + \| \eta \|_{\Sigma, \frac{7}{2}} + \| c \|_{\Sigma, \frac{5}{2}} + \| G^{3} \|_{\Sigma, \frac{3}{2}}\\
 &\lesssim& \sqrt{\bar{\mathcal{D}}} +  \sqrt{\mathcal{E}} \sqrt{\mathcal{D}}.
 \end{eqnarray*}
Thus
\begin{equation}\label{enhance_dissipation_5}
 \norm{p}_{2} \ls \sqrt{\bar{\D}} + \sqrt{\E} \sqrt{\D}.
\end{equation}

Finally, we turn to the $c$ terms in the dissipation. Write 
\begin{equation*}
\br{c} = \frac{1}{\abs{\Sigma}} \int_\Sigma c.
\end{equation*}
Then 
\begin{equation*}
 \norm{c}_{\Sigma,0} = \sqrt{ \ns{c-\br{c}}_{\Sigma,0} + \abs{\Sigma} \abs{\br{c}}^2 } \le \norm{c-\br{c}}_{\Sigma,0} + \frac{1}{\abs{\Sigma}} \abs{\int_\Sigma c}.
\end{equation*}
Using this, the Poincar\'e inequality, and Proposition \ref{nlin_ests_c}, we find that 
\begin{equation*}
 \norm{c}_{\Sigma,0} \ls \norm{\nab_\ast c}_{\Sigma,0} + \sqrt{\E} \sqrt{\D} \ls \sqrt{\bar{\D}} + \sqrt{\E} \sqrt{\D}.
\end{equation*}
On the other hand the fifth equation in \eqref{ns_perturbed} allows us to compute
\begin{equation}
 \int_{\Sigma} \dt c = \int_{\Sigma} G^5 + \Delta_* c - c_0 \diverge_* u = \int_{\Sigma} G^5,
\end{equation}
and so again the Poincar\'e inequality and Theorem \ref{nlin_ests_G} tell us that
 \begin{eqnarray*}
 \|\dt c \|_{\Sigma,0} &\le& \big\| \dt c - \br{\dt c} \big\|_{\Sigma,0} + \frac{1}{\abs{\Sigma}}\abs{\int_{\Sigma} G^5} \\
 &\lesssim& \|\nabla_{*} \dt c \|_{\Sigma, 0} + \norm{G^5}_{\Sigma,0} \\
 &\lesssim& \sqrt{\bar{\mathcal{D}}} +\sqrt{\E} \sqrt{\D}.
\end{eqnarray*}
Therefore, we obtain
\begin{equation}\label{enhance_dissipation_6}
 \| c \|_{\Sigma,3} + \| \dt c \|_{\Sigma,1}    \ls \| c \|_{\Sigma,0} + \|\dt c \|_{\Sigma,0} +  \sum_{\substack{\abs{\alpha}\le 2 \\\alpha_0 =0}}  \| \nab_* \p^\alpha c \|_{\Sigma,0}     +  \|\nab_* \dt c \|_{\Sigma,0}
 \ls \sqrt{\bar{\mathcal{D}}} +\sqrt{\E} \sqrt{\D}.
\end{equation}

Now to deduce \eqref{enhance_dissipation_0} we sum the squares of the estimates \eqref{enhance_dissipation_1}, \eqref{enhance_dissipation_2}, \eqref{enhance_dissipation_3}, \eqref{enhance_dissipation_4}, \eqref{enhance_dissipation_5}, and \eqref{enhance_dissipation_6}.

\end{proof}

\section{Proof of main results  }\label{sec_mains}

\subsection{Boundedness and decay  }

We now combine the estimates of the previous section in order to deduce our primary a priori estimate for solutions.  It shows that under a smallness condition on the energy and a finiteness condition for the integrated dissipation, the energy decays exponentially and the dissipation integral is bounded by the initial data.

\begin{thm}\label{aprioris}
Suppose that $(u,p,\eta,c)$ solves \eqref{ns_geometric} on the temporal interval $[0,T]$.  Let $\E$ and $\D$ be as defined in \eqref{def_E} and \eqref{def_D}. Then there exists a universal constant $0 < \delta_\ast < \delta$, where  $\delta \in (0,1)$ is the universal constant given in Lemma \ref{infinity_bounds}, such that if 
\begin{equation*}
 \sup_{0\le t \le T} \E(t) \le \delta_\ast \text{ and } \int_0^T \D(t) dt < \infty,
\end{equation*}
then
\begin{equation}\label{aprioris_0}
\sup_{0\le t \le T} e^{\lambda t} \E(t) + \int_0^T \D(t)dt \ls \E(0)
\end{equation}
for all $t \in [0,T]$, where $\lambda >0$ is a universal constant.
\end{thm}
\begin{proof}
According to Theorems \ref{enhance_energy} and \ref{enhance_dissipation} we have that
\begin{equation*}
\mathcal{E} \lesssim \bar{\mathcal{E}} + \mathcal{E}^{2} \text{ and } \D \ls \bar{\D} + \E \D.
\end{equation*}
Consequently, if we choose $\delta_\ast$ sufficiently small, then we may absorb the terms $\E^2$ and $\E \D$ onto the left to deduce that 
\begin{equation}\label{aprioris_1}
\bar{\E} \le  \E \ls \bar{\E} \text{ and } \bar{\D} \le \D \ls \bar{\D} \text{ on } [0,T].
\end{equation}

Next we invoke Theorem \ref{energy_ev}, which tells us that on $[0,T]$ we have the inequality
\begin{equation*}
 \frac{d}{dt} \left(\bar{\E} -\mathcal{F} \right) +  \bar{\mathcal{D}}  \lesssim     \sqrt{\mathcal{E}}  \mathcal{D} \ls \sqrt{\E} \bar{\D},
\end{equation*}
where the last inequality follows from \eqref{aprioris_1}.  Upon further restricting $\delta_\ast$ if necessary we may absorb $\sqrt{\E} \bar{\D}$ onto the left to deduce that 
\begin{equation*}
 \frac{d}{dt} \left(\bar{\E} -\mathcal{F} \right) + \hal  \bar{\mathcal{D}}  \le 0,
\end{equation*}
which when combined with the second bound in \eqref{aprioris_1} implies that 
\begin{equation}\label{aprioris_2}
\frac{d}{dt} \left(\bar{\E} -\mathcal{F} \right) + C  \D  \le 0
\end{equation}
on $[0,T]$, for $C>0$ a universal constant.

Now we turn our attention to $\mathcal{F}$.  The second estimate in \eqref{nlin_ests_F_02} of Theorem \ref{nlin_ests_F}, together with \eqref{aprioris_1}, tell us that 
\begin{equation*}
 \abs{\mathcal{F}} \ls \E^{3/2} \ls \bar{\E} \sqrt{\E},
\end{equation*}
and so if we further restrict $\delta_\ast$ we may conclude that 
\begin{equation}\label{aprioris_3}
 \hal \bar{\E} \le \bar{\E} -\mathcal{F}  \le \frac{3}{2} \bar{\E}
\end{equation}
on $[0,T]$.  In particular this tells us that $\bar{\E} - \mathcal{F} \ge 0$.

We may then integrate \eqref{aprioris_2} in time to deduce that 
\begin{equation*}
 C\int_0^T \D(t) dt \le (\bar{\E}(T) - \mathcal{F}(T)) + C \int_0^T \D(t) dt \le (\bar{\E}(0) - \mathcal{F}(0)),
\end{equation*}
from which we deduce that 
\begin{equation}\label{aprioris_4}
 \int_0^T \D(t) dt \ls \E(0).
\end{equation}
 
On the other hand, we have the obvious bound $\bar{\E} \ls \D$, and so \eqref{aprioris_3} implies that 
\begin{equation*}
 0 \le \bar{\E} - \F \ls \bar{\E} \ls \D,
\end{equation*}
and hence \eqref{aprioris_2} tells us that 
\begin{equation*}
\frac{d}{dt} \left(\bar{\E} -\mathcal{F} \right) + \lambda  \left(\bar{\E} -\mathcal{F} \right)  \le 0
\end{equation*}
for some universal constant $\lambda>0$.  Gronwall's inequality and \eqref{aprioris_3} then imply that 
\begin{equation*}
 \bar{\E}(t) \ls (\bar{\E}(t) - \mathcal{F}(t)) \ls e^{-\lambda t} (\bar{\E}(0) - \mathcal{F}(0)) \ls e^{-\lambda t} \E(0)
\end{equation*}
for all $t \in [0,T]$, and hence 
\begin{equation}\label{aprioris_5}
\sup_{0\le t \le T} e^{\lambda t} \E(t) \ls \E(0). 
\end{equation}

Now to conclude that the estimate \eqref{aprioris_0} holds we simply sum \eqref{aprioris_4} and  \eqref{aprioris_5}.

\end{proof}

\subsection{Global well-posedness  }

We now couple to the local well-posedness to produce global-in-time solutions that decay to equilibrium exponentially fast.

\begin{proof}[Proof of Theorem \ref{gwp_intro}]

First note that given $u_0, \eta_0$, and $\tilde{c}_0$, in the local existence result, Theorem \ref{lwp_intro}, we construct the remaining initial data $\dt u(\cdot,0)$, $\dt \eta(\cdot,0)$, $\dt c(\cdot,0)$, and $p(\cdot,0)$ in such a way that 
\begin{equation}\label{gwp_intro_1}
 \mathcal{E}(0) \le C_0 \left( \ns{u_0}_{H^2(\Omega)} + \ns{\eta_0}_{H^3(\Sigma)} + \ns{\tilde{c}_0-c_0}_{H^2(\Sigma)}  \right)
\end{equation}
for some universal constant $C_0 >0$.  

Let $T=1$ and choose $\delta_\ast >0$ as in Theorem \ref{aprioris}.  Choose $\kappa>0$ as in Theorem \ref{lwp_intro} and let $C_1 >0$ denote the universal constant appearing on the right side of \eqref{lwp_intro_0}.  Also let $C_2 >0$ be the universal constant appearing on the right side of \eqref{aprioris_0} and $\lambda >0$ be the universal constant appearing on the left.  Set 
\begin{equation*}
 \kappa_\ast = \frac{1}{(1+C_0)(1+C_1)(1+C_2)}\min\{ \kappa, \delta_\ast \}
\end{equation*}
and assume that \eqref{gwp_intro_00} is satisfied with $\kappa_\ast$.

Due to \eqref{gwp_intro_1}, the unique solution on $[0,1]$ produced by Theorem \ref{lwp_intro} then satisfies 
\begin{equation*}
 \sup_{0\le t \le 1} \E(t) + \int_0^1 \D(t)dt  +
  \int_0^1 \ns{\dt^2 c(t)}_{H^{-1}(\Sigma)}dt + \ns{\dt^{2N+1} u}_{(\mathcal{X}_1)^*} 
 \le C_1 \E(0) \le C_0 C_1 \kappa_\ast \le \delta_\ast.
\end{equation*}
Consequently, we may apply Theorem \ref{aprioris} to see that 
\begin{equation*}
\sup_{0\le t \le 1} e^{\lambda t} \E(t) + \int_0^1 \D(t)dt \le C_2 \E(0) \le C_0 C_2 \kappa_\ast
\end{equation*}
which in particular means that 
\begin{equation}\label{gwp_intro_2}
 \E(1) \le e^{-\lambda} C_0 C_2 \kappa_\ast \le \kappa.
\end{equation}

Due to \eqref{gwp_intro_2} we may apply Theorem \ref{lwp_intro} with initial data $u(\cdot,1),\eta(\cdot,1)$, etc, to uniquely extend the solution to $[1,2]$ in such a way that 
\begin{equation*}
 \sup_{1\le t \le 2} \E(t) + \int_1^2 \D(t)dt     + \int_1^2 \ns{\dt^2 c(t)}_{H^{-1}(\Sigma)}dt + \ns{\dt^{2N+1} u}_{(\mathcal{X}_{1,2})^*} \le C_1 \E(1) \le e^{-\lambda} C_0 C_1 C_2 \kappa_\ast \le \delta_\ast,
\end{equation*}
where $\mathcal{X}_{a,b}$ means \eqref{X_space_def} with the temporal interval replaced with $[a,b]$ in place of $[0,T]$.  We may then apply the a priori estimate of Theorem \ref{aprioris} to see that 
\begin{equation*}
\sup_{0\le t \le 2} e^{\lambda t} \E(t) + \int_0^2 \D(t)dt \le C_2 \E(0) \le C_0 C_2 \kappa_\ast
\end{equation*}
and hence that
\begin{equation*}
 \E(2) \le e^{-2 \lambda} C_0 C_2 \kappa_\ast.
\end{equation*}

We may continue iterating the above argument to ultimately deduce that the solution exists on $[0,\infty)$ and obeys the estimate \eqref{gwp_intro_01}.

\end{proof}

\appendix

\section{Surface differential operators }\label{app_surf}

\subsection{Basics }

Given a vector $X \in \mathbb{R}^3$ we write $X_\ast \in \mathbb{R}^2$ for its horizontal component, i.e. $X_\ast = X_1 e_1 + X_2 e_2$.  In same vein we write 
\begin{equation*}
 \nab_\ast f = \p_1 f e_1 + \p_2 f e_2  
\end{equation*}
for the ``horizontal'' gradient.  We also write 
\begin{equation*}
 \diverge_\ast X = \p_1 X_1 + \p_2 X_2
\end{equation*}
for the horizontal divergence operator.

The unit normal on $\Gamma(t)$ is defined via
\begin{equation*}
 \nu = \frac{1}{\sqrt{1+\abs{\nab_\ast \eta}}} (-\nab_\ast \eta,1).
\end{equation*}
We define the differential operator $\nab_\Gamma$ via $\nab_\Gamma f = \p_{\Gamma,i} f e_i$, where we define 
\begin{equation*}
 \p_{\Gamma,i} = 
\begin{cases}
\p_i - \nu_i (\nu_\ast \cdot \nab_\ast ) & \text{if } i=1,2 \\
- \nu_3 (\nu_\ast \cdot \nab_\ast )& \text{if } i=3.
\end{cases}
\end{equation*}
For a vector field $X : \Sigma \to \mathbb{R}^3$ we set 
\begin{equation*}
 \diverge_\Gamma X = \p_{\Gamma,i} X_i.
\end{equation*}

Suppose now that $f: \Gamma(t) \to \mathbb{R}$ and $X: \Gamma(t) \to \mathbb{R}^3$.  Let  $D_{\Gamma}$ denotes the intrinsic surface gradient: $D_\Gamma f : \Gamma(t) \to \mathbb{R}^3$ such that $D_\Gamma f(x)$ is perpendicular to $T_x \Gamma(t)$ for each $x \in \Gamma(t)$.  Also let $\text{Div}_\Gamma X$ denote the intrinsic surface divergence of $X$.  These quantities are related to the above defined ones as follows.  If we write
\begin{equation*}
 f \circ \eta = f(x_\ast, \eta(x_\ast,t)) \text{ and } X \circ \eta = X(x_\ast,\eta(x_\ast,t))
\end{equation*}
then $f \circ \eta: \Sigma \to \mathbb{R}$ and $ X\circ \eta : \Sigma \to \mathbb{R}^3$, and 
\begin{equation*}
 D_\Gamma f (x_\ast,\eta(x_\ast,t)) = \nab_\Gamma ( f \circ \eta) (x_\ast) \text{ and } \text{Div}_\Gamma X (x_\ast,\eta(x_\ast,t)) = \diverge_\Gamma X (x_\ast) \text{ for all } x_\ast \in \Sigma.
\end{equation*}
In other words, $\nab_\Gamma$ and $\diverge_\Gamma$ are the manifestations of the surface gradient and divergence when functions and vector fields are pulled back to $\Sigma$.

The operators $D_\Gamma$ and $\text{Div}_\Gamma$ are known to satisfy a number of useful identities.  Here we record the versions of these identities for $\nab_\Gamma$ and $\diverge_\Gamma$.  We begin with some preliminary calculations.

\begin{lem}\label{surf_idents}
We have the following identities:
\begin{equation}\label{surf_idents_01}
 \diverge_\Gamma \nu = \p_1 \nu_1 + \p_2 \nu_2 = \diverge_\ast \nu_\ast= - \diverge_\ast \left(\frac{\nab_\ast \eta}{\sqrt{1+\abs{\nab_\ast \eta}^2}}  \right) = - H,
\end{equation}
\begin{equation}\label{surf_idents_02}
 \p_i \sqrt{1+\abs{\nab_\ast \eta}^2 } = - \nu_\ast \cdot \nab_\ast \p_i \eta,
\end{equation}
\begin{equation}\label{surf_idents_03}
 \p_t \sqrt{1+\abs{\nab_\ast \eta}^2 } = \diverge_\ast \left(\dt \eta \frac{\nab_\ast \eta}{\sqrt{1+\abs{\nab_\ast \eta}^2}} \right) - \dt \eta H,
\end{equation}
and 
\begin{equation}\label{surf_idents_04}
\nab_\Gamma f \cdot\nu =0.
\end{equation}
\end{lem}
\begin{proof}
To prove \eqref{surf_idents_01} we first use the  definition of $\diverge_\Gamma$ to write
\begin{equation*}
\diverge_\Gamma \nu = \p_1 \nu_1 + \p_2 \nu_2  - \nu_i (\nu_\ast \cdot \nab_\ast \nu_i).
\end{equation*}
We have that $\abs{\nu}^2=1$ on $\Sigma$, so 
\begin{equation*}
 0 = \nu_\ast \cdot \nab_\ast 1 = \nu_\ast \cdot \nab_\ast \abs{\nu}^2 = 2 \nu_i \nu_\ast \cdot \nab_\ast \nu_i,
\end{equation*}
Upon combining these two calculations we deduce \eqref{surf_idents_01}.

For \eqref{surf_idents_02} we compute
\begin{equation*}
 \p_i \sqrt{1+\abs{\nab_\ast \eta}^2 } = \frac{\p_j \eta}{\sqrt{1+\abs{\nab_\ast \eta}^2 }} \p_j \p_i \eta = - \nu_j \p_j \p_i \eta =- \nu_\ast \cdot \nab_\ast \p_i \eta.
\end{equation*}
The identity \eqref{surf_idents_03} follows from a similar computation.  

The identity \eqref{surf_idents_04} follows from the fact that $\nab_\Gamma f = \nab_\ast f - \nu (\nu_\ast \cdot \nab_\ast f)$, and hence 
\begin{equation*}
 \nab_\Gamma f \cdot \nu = \nab_\ast f \cdot \nu - \abs{\nu}^2 (\nu_\ast \cdot \nab_\ast f) = \nu_\ast \cdot \nab_\ast f - \nu_\ast \cdot \nab_\ast f =0.
\end{equation*}
\end{proof}

Next we record the key integration by parts identities.

\begin{prop}\label{surf_ibp}
We have the following identities for $f,g : \Sigma \to \mathbb{R}$ and $X : \Sigma \to \mathbb{R}^3$:
\begin{equation}\label{surf_ibp_01}
 \int_{\Sigma} \p_{\Gamma,i} f g \sqrt{1+\abs{\nab_\ast \eta}^2 } =  \int_{\Sigma} -\left( f \p_{\Gamma,i} g + fg \nu_i H  \right) \sqrt{1+\abs{\nab_\ast \eta}^2 }
\end{equation}
and 
\begin{equation}\label{surf_ibp_02}
 \int_\Sigma \diverge_\Gamma X \sqrt{1+\abs{\nab_\ast \eta}^2 } = \int_\Sigma - X \cdot \nu  H \sqrt{1+\abs{\nab_\ast \eta}^2 }.
\end{equation}
\end{prop}
\begin{proof}
The identity \eqref{surf_ibp_02} follows immediately from \eqref{surf_ibp_01}, so it suffices to prove \eqref{surf_ibp_01}.

Assume initially that $i=1,2$.   Standard integration by parts reveals that
\begin{multline*} 
\int_\Sigma \p_{\Gamma,i} f g  \sqrt{1+\abs{\nab_\ast \eta}^2 } = -\int_\Sigma - f \p_i g \sqrt{1+\abs{\nab_\ast \eta}^2 } - fg \p_i \sqrt{1+\abs{\nab_\ast \eta}^2 } + \diverge_\ast \left(\nu_\ast \nu_i g   \sqrt{1+\abs{\nab_\ast \eta}^2 }\right) \\
=-\int_\Sigma f \p_{\Gamma,i} g \sqrt{1+\abs{\nab_\ast \eta}^2 } + fg \left[(\diverge_\ast \nu_\ast) \nu_i \sqrt{1+\abs{\nab_\ast \eta}^2} +   \nu_\ast \cdot \nab_\ast \left( -\p_i \eta  \right) - \p_i  \sqrt{1+\abs{\nab_\ast \eta}^2 }  \right] \\
=-\int_\Sigma (f \p_{\Gamma,i} g  + fg \nu_i H)\sqrt{1+\abs{\nab_\ast \eta}^2 },
\end{multline*}
where in the last line we have used the identities of Lemma \ref{surf_idents}.  This proves  \eqref{surf_ibp_01}  when $i=1,2$.

Now assume that $i=3$.  Since $\p_{\Gamma,3} = -\nu_3 (\nu_\ast \cdot \nab_\ast)$ we may then compute 
\begin{multline*}
 \int_\Sigma \p_{\Gamma,3} f g  \sqrt{1+\abs{\nab_\ast \eta}^2 } = \int_\Sigma f \diverge_\ast \left(\nu_\ast \nu_3 g \sqrt{1+\abs{\nab_\ast \eta}^2 } \right) = \int_\Sigma -f \p_{\Gamma,3} g \sqrt{1+\abs{\nab_\ast \eta}^2 } \\ + \int_\Sigma fg \diverge_\ast \left(\nu_\ast \nu_3 \sqrt{1+\abs{\nab_\ast \eta}^2 } \right) 
 = \int_\Sigma -f \p_{\Gamma,3} g \sqrt{1+\abs{\nab_\ast \eta}^2 } + \int_\Sigma fg \diverge_\ast \left(\nu_\ast  \right) \\
 =  \int_\Sigma - (f \p_{\Gamma,3} g + fg \nu_3 H ) \sqrt{1+\abs{\nab_\ast \eta}^2 }.
\end{multline*}
This proves \eqref{surf_ibp_01} when $i=3$.
\end{proof}

\subsection{A PDE identity }

Here we record an important identity for solutions to certain PDEs.

\begin{prop}\label{surf_c_ev}
Let $f \in C^2(\mathbb{R})$.  Suppose that $\tilde{c}$ and $\eta$ satisfy 
\begin{equation*}
\begin{cases}
\dt \tilde{c}  +u \cdot\nab_\ast \tilde{c} + \tilde{c} \diverge_\Gamma u = \gamma \Delta_\Gamma \tilde{c} \\
\dt \eta = u\cdot \nu \sqrt{1+\abs{\nab_\ast \eta}^2}.
\end{cases}
\end{equation*}
Then 
\begin{equation*}
\frac{d}{dt} \int_\Sigma f(\tilde{c}) \sqrt{1+\abs{\nab_\ast \eta}^2} = \int_\Sigma \left((f(\tilde{c}) - f'(\tilde{c}) \tilde{c} ) \diverge_\Gamma u - f''(\tilde{c}) \abs{\nab_\Gamma \tilde{c}}^2 \right)  \sqrt{1+\abs{\nab_\ast \eta}^2}
\end{equation*}
\end{prop}
\begin{proof}
We begin by computing 
\begin{equation*}
\frac{d}{dt}\int_{\Sigma} f(\tilde{c})  \sqrt{1+\abs{\nab_\ast \eta}^2} = \int_\Sigma f'(\tilde{c}) \dt \tilde{c}  \sqrt{1+\abs{\nab_\ast \eta}^2} + f(\tilde{c}) \dt  \sqrt{1+\abs{\nab_\ast \eta}^2} := I + II.
\end{equation*}

Note that 
\begin{equation*}
 \Delta_\Gamma f(\tilde{c}) = \diverge_\Gamma( \nab^\gamma f(\tilde{c})) = \diverge_\Gamma (f'(\tilde{c}) \nab_\Gamma \tilde{c}) = f'(\tilde{c}) \Delta_\Gamma \tilde{c} + f''(\tilde{c}) \abs{\nab_\Gamma \tilde{c}}^2.
\end{equation*}
Using this, we may rewrite 
\begin{multline*}
f'(\tilde{c})  \dt \tilde{c} = f'(\tilde{c}) \left( -u \cdot\nab_\ast \tilde{c} - \tilde{c} \diverge_\Gamma u + \gamma \Delta_\Gamma \tilde{c} \right)
= - u \cdot \nab_\ast f(\tilde{c}) - f(\tilde{c}) \diverge_\Gamma u + \gamma \Delta_\Gamma f(\tilde{c}) \\
+ (f(\tilde{c}) - f'(\tilde{c}) \tilde{c} ) \diverge_\Gamma u - \gamma f''(\tilde{c}) \abs{\nab_\Gamma \tilde{c}}^2.
\end{multline*}
Consequently, we may rewrite $I= I_1 + I_2$ for 
\begin{equation*}
 I_1 = \int_\Sigma \left(- u \cdot \nab_\ast f(\tilde{c}) - f(\tilde{c}) \diverge_\Gamma u + \gamma \Delta_\Gamma f(\tilde{c}) \right) \sqrt{1+\abs{\nab_\ast \eta}^2}
\end{equation*}
and 
\begin{equation*}
 I_2 = \int_\Sigma \left( (f(\tilde{c}) - f'(\tilde{c}) \tilde{c} ) \diverge_\Gamma u - \gamma f''(\tilde{c}) \abs{\nab_\Gamma \tilde{c}}^2\right)\sqrt{1+\abs{\nab_\ast \eta}^2}.
\end{equation*}
Thus, to complete the proof it suffices to show that $I_1 + II =0$.

To this end we first use Proposition \ref{surf_ibp} to compute
\begin{equation*}
 \int_\Sigma \gamma \Delta_\Gamma f(\tilde{c})   \sqrt{1+\abs{\nab_\ast \eta}^2} =  \int_\Sigma -\gamma \nab_\Gamma f(\tilde{c}) \cdot \nu H   \sqrt{1+\abs{\nab_\ast \eta}^2} =0
\end{equation*}
since Lemma \ref{surf_idents} tells us that $\nab_\Gamma f(\tilde{c}) \cdot \nu =0$.  Similarly, 
\begin{equation*}
 \int_\Sigma  -f(\tilde{c}) \diverge_\Gamma u \sqrt{1+\abs{\nab_\ast \eta}^2} =  \int_\Sigma  \left( u \cdot \nab_\Gamma f(\tilde{c}) + f(\tilde{c}) u\cdot \nu H  \right) \sqrt{1+\abs{\nab_\ast \eta}^2},
\end{equation*}
and hence
\begin{equation*}
 I_1 = \int_\Sigma \left(- u \cdot \nab_\ast f(\tilde{c}) + u \cdot \nab_\Gamma f(\tilde{c}) + f(\tilde{c})  u\cdot \nu H  \right) \sqrt{1+\abs{\nab_\ast \eta}^2}.
\end{equation*}

On the other hand, \eqref{surf_idents_03} and the equality $\dt \eta = u\cdot \nu \sqrt{1+ \abs{\nab_\ast \eta}^2}$ tell us that
\begin{equation*}
 II = \int_\Sigma f(\tilde{c}) \left( \diverge_\ast(u\cdot \nu \nab_\ast \eta )  - u\cdot \nu H  \sqrt{1+\abs{\nab_\ast \eta}^2} \right) := II_1 + II_2.
\end{equation*}
Upon integrating by parts, we may rewrite 
\begin{equation*}
 II_1 = \int_\Sigma -u\cdot \nu \nab_\ast f(\tilde{c}) \cdot \nab_\ast \eta=   \int_\Sigma u\cdot \nu \nab_\ast f(\tilde{c}) \cdot \nu_\ast \sqrt{1+ \abs{\nab_\ast \eta}^2}.
\end{equation*}
Thus 
\begin{equation*}
 II =  \int_\Sigma \left( u\cdot \nu \nab_\ast f(\tilde{c}) \cdot \nu_\ast  - f(\tilde{c}) u\cdot \nu H \right) \sqrt{1+ \abs{\nab_\ast \eta}^2}.
\end{equation*}

To conclude we first note that 
\begin{equation*}
 u \cdot \nab_\Gamma f(\tilde{c}) =   u \cdot \left[  \nab_\ast f(\tilde{c})  -  \nu (\nu_\ast \cdot \nab_\ast f(\tilde{c}))  \right].
\end{equation*}
Thus upon summing the above expressions for $I_1$ and $II$ we find that $I_1 + II =0$, which then yields the desired identity.

\end{proof}

\section{Analytic tools }\label{app_tools}

\subsection{Poisson integral}

Suppose that $\Sigma = (L_1 \mathbb{T}) \times (L_2 \mathbb{T})$.  We define the Poisson integral in $\Omega_- = \Sigma \times (-\infty,0)$ by
\begin{equation}\label{poisson_def_per}
\mathcal{P} f(x) = \sum_{n \in   (L_1^{-1} \mathbb{Z}) \times (L_2^{-1} \mathbb{Z}) }  e^{2\pi i n \cdot x_\ast} e^{2\pi \abs{n}x_3} \hat{f}(n),
\end{equation}
where for $n \in   (L_1^{-1} \mathbb{Z}) \times (L_2^{-1} \mathbb{Z})$ we have written
\begin{equation*}
 \hat{f}(n) = \int_\Sigma f(x_\ast)  \frac{e^{-2\pi i n \cdot x_\ast}}{L_1 L_2} dx_\ast.
\end{equation*}
It is well known that $\mathcal{P}: H^{s}(\Sigma) \rightarrow H^{s+1/2}(\Omega_-)$ is a bounded linear operator for $s>0$.

\begin{lem}\label{p_poisson}
Let $\mathcal{P} f$ be the Poisson integral of a function $f$ that is either in $\dot{H}^{q}(\Sigma)$ or $\dot{H}^{q-1/2}(\Sigma)$ for $q \in \mathbb{N}$.  Then
\begin{equation*} 
 \ns{\nab^q \mathcal{P}f }_{0} \ls \norm{f}_{\dot{H}^{q-1/2}(\Sigma)}^2 \text{ and }  \ns{\nab^q \mathcal{P}f }_{0} \ls \norm{f}_{\dot{H}^{q}(\Sigma)}^2.
\end{equation*}
\end{lem}
\begin{proof}
This is proved, for instance, in Lemma A.3 of \cite{GT2}.
\end{proof}

We will also need $L^\infty$ estimates.
\begin{lem}\label{p_poisson_2}
Let $\mathcal{P} f$ be the Poisson integral of a function $f$ that is in $\dot{H}^{q+s}(\Sigma)$ for $q\ge 1$ an integer and $s> 1$.  Then
\begin{equation*} 
 \pns{\nab^q \mathcal{P}f }{\infty} \ls \ns{f}_{\dot{H}^{q+s}}.
\end{equation*}
The same estimate holds for $q=0$ if $f$ satisfies $\int_{\Sigma} f =0$.
\end{lem}
\begin{proof}
This is proved, for instance, in Lemma A.4 of \cite{GT2}.
\end{proof}

\subsection{Product estimates}
  
The following lemma is key for nonlinear estimates.  
  
\begin{lem}\label{product_ests}
The following hold.
\begin{enumerate}
 \item Let $0\le r \le s_1 \le s_2$ with $s_1 > n/2$. If $f\in H^{s_1}$, $g\in H^{s_2}$ then $fg \in H^r$ and
\begin{equation}\label{i_s_p_01}
\|{fg}\|_{r} \  \lesssim  \ \|{f}\|_{{s_1}} \|{g}\|_{{s_2}}.
\end{equation}

\item Let $0\le r \le s_1 \le s_2$ with $s_2 >r+ n/2$. If $f\in H^{s_1}$, $g\in H^{s_2}$ then $fg \in H^r$ and
\begin{equation}\label{i_s_p_02}
 \|{fg}\|_{r} \  \lesssim \  \|{f}\|_{{s_1}}  \|{g}\|_{{s_2}}.
\end{equation}

\item Let $0\le r \le s_1 \le s_2$ with $s_2 >r+ n/2$. If $f \in H^{-r}(\Sigma),$ $g \in H^{s_2}(\Sigma)$ then $fg \in H^{-s_1}(\Sigma)$ and
\begin{equation*}
 \|{fg}\|_{\Sigma,-s_1}  \ \ls  \ \|{f}\|_{\Sigma,-r}  \|{g}\|_{\Sigma,s_2}.
\end{equation*}
\end{enumerate}
\end{lem}
\begin{proof}
See, for instance, the appendix of \cite{GT2}. 
\end{proof}

\textbf{Acknowledgments:}  C. Kim thanks the Center for Nonlinear Analysis at Carnegie Mellon University for the kind hospitality during his stay.


\end{document}